\documentclass[11pt]{article}
\usepackage{amsmath,amssymb,amsthm}
\usepackage{longtable}
\usepackage{a4wide}
\usepackage{graphicx}
\newcommand{\xa}{x_\alpha}
\newcommand{\xb}{x_\beta}
\newcommand{\xab}{x_{\alpha+\beta}}
\newcommand{\xaa}{x_{2\alpha+\beta}}

\newtheorem*{theorem*}{Theorem 1}
\newtheorem*{theorem**}{Theorem 2}

\newtheorem{defn}{Definition}[section]
\newtheorem{theorem}[defn]{Theorem}
\newtheorem{remark}[defn]{Remark}
\newtheorem{lemma}[defn]{Lemma}

\newtheorem{cor}[defn]{Corollary}
\newtheorem{claim}[defn]{Claim}
\newtheorem{prop}[defn]{Proposition}

\title{Presentations: from Kac-Moody  groups to profinite  and back
}

\author{Inna Capdeboscq, Alexander Lubotzky and Bertrand R{\'e}my }

\date{\today}

\begin{document}
\maketitle

\vspace{5cm}

\hrule

\vspace{0.5cm}

{\small
\noindent
{\bf Abstract:}
We go back and forth between, on the one hand, presentations of arithmetic and Kac-Moody groups and, on the other hand, presentations of profinite groups, deducing along the way new results on both.
\vspace{0,2cm}

\noindent
{\bf Keywords:}
arithmetic and profinite groups, bounded presentation, low-dimensional cohomology.

\vspace{0,1cm}

\noindent
{\bf AMS classification (2000):}
11F06,
17B22,
20E32,
20F20, 
20G44,
20H05,
22E40, 
51E24.
}

\vspace{0,5cm}
\hrule
\newpage


\section*{Introduction}

In recent years several papers have been dedicated to showing that finite simple groups or arithmetic groups have presentations with bounded number  of generators and relations \cite{kn::KL}, \cite{kn::GKKL1}, \cite{kn::GKKL2}, \cite{kn::GKKL3}, \cite{kn::Cap}.
This paper is about such quantitative results for some classes of profinite groups.

More precisely, we will give essentially optimal bounds on the size of the 
presentations of maximal compact subgroups and of maximal pro-$p$ subgroups of simple Chevalley groups over local fields.
These results can be also expressed as bounds on 2-cohomology groups -- see below. 

Let ${\bf G}$ be a simple, simply connected, Chevalley group scheme of Lie rank $l$ and let $q$ be a power $p^a$ for some prime $p$ and some exponent $a \geqslant 1$.
The non-Archimedean Lie group ${\bf G}\bigl( \mathbf{F}_q (\!( t )\!) \bigr)$ acts simplicially on its Bruhat-Tits building \cite{TitsCorvallis}.
In this action, each facet stabilizer is an extension of a finite group of Lie type by a pro-$p$ group. 
Combining this with a non-positive curvature argument implies that maximal pro-$p$ subgroups in ${\bf G}\bigl( \mathbf{F}_q (\!( t )\!) \bigr)$ are all conjugate to one another \cite[1.C.2]{RemyGAFA} (while this is not the case for maximal compact subgroups); we henceforth call {\it pro-$p$ Sylow subgroup}~a maximal pro-$p$ subgroup of ${\bf G}\bigl( \mathbf{F}_q (\!( t )\!) \bigr)$.
If $P$ is such a subgroup then, up to conjugating it, we can assume that we have $P < {\bf G}(\mathbf{F}_q[[t]]) < {\bf G}\bigl(\mathbf{F}_q(\!(t)\!)\bigr)$.
The group ${\bf G}(\mathbf{F}_q[[t]])$ is a {\it special}~maximal compact subgroup.

\begin{theorem}
\label{th::presentation_max_compact}
There exists a constant $C>0$ such that for any simple, simply connected, Chevalley group ${\bf G}$ of rank $\geqslant 2$ and for any prime power $q=p^a\geqslant 4$, the group $G={\bf G}(\mathbf{F}_q[[t]])$ admits a profinite presentation  $\Sigma(G)$, with $D_{\Sigma(G)}$ generators and $R_{\Sigma(G)}$ relations, satisfying 

\smallskip 

\centerline{$D_{\Sigma(G)}+R_{\Sigma(G)} \leqslant C$.}
\end{theorem}

For a group $X$, let $d(X)$ (respectively $r(X)$) denote the minimal number of generators (respectively relations) of $X$ in  any of its profinite or discrete presentation (depending on whether $X$ is  a profinite or a discrete group).
Then  Theorem \ref{th::presentation_max_compact} implies that in particular $$d(G)+r(G)\leqslant C.$$ 

Theorem \ref{th::presentation_max_compact} is proved by combining the results of \cite{kn::Cap} on bounded presentations of the discrete Kay-Moody group ${\bf G}(\mathbf{F}_q[t,t^{-1}])$, with the fact that it has the congruence subgroup property. 
It is important at this point to note that the above uniformness statement is {\it not}~true globally, that is for abstract presentations of arithmetic groups obtained by replacing the local rings $\mathbf{F}_q[[t]]$ by rings of integers of global fields (see  \S3 below). 

Now, $P$  as above is a subgroup of index $q^{O(l^2)}$ of ${\bf G}(\mathbf{F}_q[[t]])$, from which one deduces, using the Reidemeister-Schreier theorem, that $P$ has a presentation $\Sigma(P)$ satisfying
$$D_{\Sigma(P)}+R_{\Sigma(P)} \leqslant q^{O(l^2)}.$$ 

\smallskip

By different methods, we obtain a presentation of the pro-$p$ group $P$ with much better bounds. 
A word of explanation is needed here. Let $r_p(P)$ be the minimal number of relations needed to define $P$ as a  pro-$p$ group. 
In \cite[Corollary 5.5]{kn::L} it is shown that we have:
$$r(P)=\max\{d(P), r_p(P)\}.$$

%

\begin{theorem}
\label{th::rough_presentation_Sylow}
Let ${\bf G}$ be a simple, simply connected, Chevalley group  scheme of rank $l$.
Let $P$ be a Sylow pro-$p$-subgroup of $G={\bf G}\bigl(\mathbf{F}_q(\!(t)\!)\bigr)$ where $q=p^a$.
Then at least for $l \geqslant 3$
and $q \geqslant 16$,  we obtain that $d(P)=a(l+1)$, so in particular $d(P)$ is linear in $a$ and in $l$.
On the other hand, $r_p(P)=r(P)$ is bounded from below and from above by polynomials of degree $2$ in $a$ and in $l$; these polynomials do not depend on $p$.
\end{theorem}

The precise formula on the number of generators is deduced from Kac-Moody theory \cite[Corollary 2.5]{kn::CR}.
The lower bound on the number of relations is derived from  the Golod-Shafarevich inequality while the upper bound is again deduced from the theory of Kac-Moody groups.
More precisely, a discrete Kac-Moody group $\Gamma_0$ is chosen carefully so that $P$ is equal to the pro-$p$ completion of a well-understood subgroup $\Gamma$ of $\Gamma_0$. The presentation of $\Gamma$ in the discrete category serves also as a presentation for $P$ in the pro-$p$ category. 

The method of the proof  enables us to go also backward and to deduce from the Golod-Shafarevich inequality that the stated presentation of $\Gamma$ is essentially optimal (see Corollary~\ref{cor::2.5}).

Recall that for a profinite group the number of relations is expressed by the dimensions of the $2$nd cohomology groups of various modules \cite[Corollary 5.5, 5.6]{kn::L}.
It gets a particular nice form for a pro-$p$ group $P$, where $r_p(P)={\rm dim}_{{\bf Z}/p{\bf Z}} {\rm H}^2(P,\mathbf{Z}/p\mathbf{Z})$.
Thus combining Theorems \ref{th::presentation_max_compact} and \ref{th::rough_presentation_Sylow} we obtain

\begin{cor}
\label{cor::combined}
Let ${\bf G}$ be a simple, simply connected, Chevalley group scheme of rank $l$.
Let $P$ be a Sylow pro-$p$-subgroup in $G={\bf G}\bigl(\mathbf{F}_q[[t]]\bigr)$ with $q=p^a$.
Then the following conditions hold:
\begin{enumerate}
\item If $l\geqslant 2$ and $q\geqslant 4$, then ${\rm dim}_{\mathbf{F}_s} {\rm H}^2(G,M)\leqslant C\cdot {\rm dim}_{\mathbf{F}_s} M$ for every simple $\mathbf{F}_s[G]$-module $M$ and every prime $s$.
\item If $l\geqslant 3$
and $q\geqslant 16$, then $\dim_{{\bf Z}/p{\bf Z}} {\rm H}^2(P,\mathbf{Z}/p\mathbf{Z})$  is bounded from below and from above by polynomials of degree $2$ in $a$ and in $l$; these polynomials do not depend on $p$.
\end{enumerate}
\end{cor}


\smallskip

Theorems  \ref{th::presentation_max_compact} and \ref{th::rough_presentation_Sylow} suggest that similar results are valid also in characteristic $0$, but our methods are not that efficient there (see $\S$3).
There are two main differences between the characteristic zero case and the positive characteristic case.

The first one is that in characteristic $p$ $\it{all}$ local fields are obtained as completions of $\it{one}$ global field $\mathbf{F}_p(t)$, while in characteristic $0$, there is no such global field.
In our method of proof, which goes from global to local, this difference is crucial: we prove a uniform result for all groups defined over a given global field $k$ (Theorem \ref{theorem::3.1}), but there is no uniform result over all global fields (see Remark~\ref{remark::34}).
Our uniform result for local fields in characteristic $p$ uses substantially  the fact that they all are completions of one global field.

Secondly, the pro-$p$ Sylow subgroup of ${\bf G}\bigl(\mathbf{F}_q(\!(t)\!)\bigr)$ is the pro-$p$ completion  of a suitable subgroup of a Kac-Moody group (see Lemma~\ref{key_lemma} below for the exact result) and we are making a crucial use of this fact. No such result is known to us in characteristic $0$.

Still one has good reasons to believe that results like Theorem 0.1 and Theorem 0.2 are valid also in characteristic $0$. In Section 3, we sketch few partial results in this direction.

\smallskip

The structure of the paper is as follows.
Section \ref{s - maximal compact} deals with uniformly bounded profinite presentations of maximal compact subgroups of non-archimedean Chevalley groups in characteristic $p$.
Section \ref{s - Sylow} investigates more carefully the maximal pro-$p$ subgroups of these groups.
Section 3 goes back to the questions of $\S$\ref{s - maximal compact} in the characteristic 0 case.
The proof for the global case is given in $\S$4 and uses K-theory.

\section{Presentations of $\mathbf{F}_q[[t]]$-points of split simple groups}
\label{s - maximal compact}
 
Let $q=p^a$ with $p$ a prime and $a \geqslant 1$.
Let ${\bf G}$ be a simple, simply connected, Chevalley group.
In this section we are interested in controlling presentations of the virtually pro-$p$ groups ${\bf G}(\mathbf{F}_{q^e}[[t]])$ for arbitrary exponent $e \geqslant 1$. 
The general idea of this section is that ${\bf G}(\mathbf{F}_{q^e}[[t]])$ appears as a factor of the profinite completion of  ${\bf G}(\mathbf{F}_q[t,t^{-1}])$ and that a presentation of the latter discrete group naturally gives a profinite presentation of this completion.

Therefore, in the first half of the argument we are only interested in ${\bf G}(\mathbf{F}_q[t,t^{-1}])$.
Usually, the main viewpoint on the group ${\bf G}(\mathbf{F}_q[t,t^{-1}])$ is as an $S$-arithmetic group \cite{Margulis}. 
It provides a lot of information (cf. \cite{PlatRap}), but another possibility is to see it as a (split) Kac-Moody group of affine type.
This gives additional information of combinatorial nature.
Indeed, this group acts on a product of two twin buildings \cite[12.6.3]{RemyAst} and using the action on a single building, we can deduce from this a suitable amalgamation theorem for ${\bf G}(\mathbf{F}_q[t,t^{-1}])$. 
We use here a combination of the two viewpoints.

\subsection{Bounded presentation for affine Kac-Moody groups over finite fields}

To begin with, let us recall  that by  \cite[Theorem 2.1]{kn::Cap} there exists a so-called {\it bounded}~family of presentations for affine Kac-Moody groups (with the exception of $\widetilde{A}_1$ and  ${^*\widetilde{A}}_1$). In particular, we have: 

\begin{theorem}[\cite{kn::Cap}]
\label{theorem::C}
 There exists a constant $C>0$ such that for any simple, simply connected, Chevalley group scheme ${\bf G}$ of rank $\geqslant 2$ and for any prime power $q \geqslant 4$, the group $G={\bf G}(\mathbf{F}_q[t,t^{-1}])$ admits a presentation  $\Sigma(G)$ with $D_{\Sigma(G)}$ generators and $R_{\Sigma(G)}$ relations satisfying 

\smallskip 

\centerline{$D_{\Sigma(G)}+R_{\Sigma(G)} \leqslant C$.}
\end{theorem}


\subsection{Bounded presentation for profinite Chevalley groups in characteristic~$p$}

We can now turn to completion processes. 
The following proposition provides a relationship between presentations of ${\bf G}(\mathbf{F}_q[t,t^{-1}])$ and of ${\bf G}(\mathbf{F}_q[[t]])$; it is a quantitative version of a method already used in \cite{LubPres}. 

\begin{prop}
\label{prop::AL}
Assume that the rank of the simple, simply connected, Chevalley group ${\bf G}$ is $\geqslant 2$ and that the arithmetic group ${\bf G}(\mathbf{F}_q[t, t^{-1}])$  has a presentation with $d$ generators and $r$ relations.  
Then for any $e \geqslant 1$, the group ${\bf G}(\mathbf{F}_{q^e}[[t]])$ has a profinite presentation with $d$ generators and $r+1$ relations.
\end{prop}

\begin{proof}
Set $A=\mathbf{F}_q[t,t^{-1}]$ and denote by $\widehat{A}$ the profinite completion of this ring. 
Let  $\mathcal{P}$ be the set of monic irreducible polynomials in $\mathbf{F}_q[t]$ and 
$\mathcal{P}_e$ those of degree $e$.
The prime ideals in $A=\mathbf{F}_q[t][{1 \over t}]$ are parametrized by $\mathcal{P} \setminus \{ t \}$, so by the Chinese remainder theorem, we have

\smallskip 

\centerline{$\widehat{A} \, \cong \prod_{f \in \mathcal{P} \setminus \{ t \}} \widehat{A}^f$}

\smallskip  

\noindent where $\widehat{A}^f = \varprojlim_{n \geqslant 1} A/(f^n)$ is isomorphic to $\mathbf{F}_{q^{{\rm deg}(f)}}[[x]]$.
Using $\mathcal{P} = \bigsqcup_{e \geqslant 1} \mathcal{P}_e$, we see therefore that

\smallskip 

\centerline{$\widehat{A} \, \cong \, (\mathbf{F}_q[[t]])^{q-1} \times \prod_{e \geqslant 2} (\mathbf{F}_{q^e}[[t]])^{\# \mathcal{P}_e}$.}

\medskip  

We set $S = \{ 0;\infty \}$. The group ${\bf G}(A)$ is thus an $S$-arithmetic group. We denote by $\widehat{{\bf G}(A)}$ its profinite completion and by $\overline{{\bf G}(A)}$ its $S$-congruence completion.
The canonical map $\pi : \widehat{{\bf G}(A)} \to \overline{{\bf G}(A)}$ is a continuous group homomorphism restricting to the identity map of the dense subgroup ${\bf G}(A)$; as a consequence, $\pi$ is surjective.
Since there is no place of ${\bf F}_q(t)$ at which ${\bf G}$ is anisotropic and since ${\bf G}$ is simply connected, strong approximation \cite{Prasad} implies that the $S$-congruence completion $\overline{{\bf G}(A)}$ is described by means of 
the $S$-ad\`eles $\mathbf{A}_S$ of ${\bf F}_q(t)$. One therefore deduces:
\smallskip

\centerline{$\overline{{\bf G}(A)} \, \cong \, {\bf G}(\widehat{A}) \, \cong \, {\bf G}(\mathbf{F}_q[[t]])^{q-1} \times \prod_{e \geqslant 2} {\bf G}(\mathbf{F}_{q^e}[[t]])^{\# \mathcal{P}_e}$.}

\smallskip

For further use in the paper, let us quote a precise statement on the solution to the congruence subgroup problem (see \cite{RaghunathanCSP}, \cite{PraRag} and \cite[VIII.2.16 Theorem]{Margulis}). 

\begin{theorem}
\label{th:CSP} 
Let ${\bf G}$ be a simply connected Chevalley group, let $k$ be a global field and let $S$ be a finite set of places of $k$. 
Let $A$ denote the ring of $S$-integers in $k$ and for each $v \in S$, let $k_v$ be the corresponding completion of $k$. 
We set ${\rm rk}_S({\bf G}) = \sum_{v \in S} {\rm rk}_{k_v}({\bf G})$ and assume that ${\rm rk}_S({\bf G}) \geqslant 2$. 
Then the kernel of the map $\pi : \widehat{{\bf G}(A)} \to \overline{{\bf G}(A)}$ is cyclic; it is even trivial whenever $S$ contains a non-Archimedean place. 
\end{theorem}

In our case, where $G$ has rank $\geqslant 2$ and all places are non-Archimedean, the map $\pi$ is an isomorphism~: $\widehat{{\bf G}(A)} \cong {\bf G}(\widehat{A})$.
This implies that  for any fixed $e \geqslant 1$, 
one can write $\widehat{{\bf G}(A)}$ as a product:
\smallskip

\centerline{$\widehat{{\bf G}(A)} = {\bf G}(\mathbf{F}_{q^e}[[t]])\times M$}

\smallskip

\noindent where $M$ is a product of infinitely many groups of the form ${\bf G}(\mathbf{F}_{q^s}[[t]])$ with finite multiplicity for each exponent $s \geqslant 1$.

By our group-theoretic hypothesis, the abstract group ${\bf G}(A)$ has a presentation with $d$ generators and $r$ relations, therefore so does $\widehat{{\bf G}(A)}$ as a profinite group.
Now, notice that $M$, as a normal subgroup of $\widehat{{\bf G}(A)}$, is generated by one element!
Indeed, every factor of the direct product $M$, say ${\bf G}(\mathbf{F}_{q^s}[[t]])$, has no non-trivial abelian quotient; moreover it admits a semi-direct product decomposition ${\bf G}(\mathbf{F}_{q^s}) \ltimes Q$ where $Q$ is a pro-$p$ group \cite[1.C]{RemRon}. 
The semi-direct product decomposition implies that ${\bf G}(\mathbf{F}_{q^s}[[t]])$ has a unique maximal normal subgroup.
Indeed, let $\pi : {\bf G}(\mathbf{F}_{q^s}[[t]]) \to S$ be a quotient map to a simple (necessarily non-abelian) group.
The group $Q$ is normal in ${\bf G}(\mathbf{F}_{q^s}[[t]])$, therefore $\pi(Q)$ is a normal $p$-group in $S$, hence is trivial. 
This implies that $\pi$ factorizes through the homomorphism ${\bf G}(\mathbf{F}_{q^s}[[t]]) \to {\bf G}(\mathbf{F}_{q^s})$, whose target group has a unique simple quotient.

Pick then $g \in M$ such that each one of its coordinates is outside that maximal normal subgroup; it is not difficult to see that the normal closure $\langle\!\langle g \rangle\!\rangle_M$ of $g$ in $M$ is $M$ itself.
Thus we get a profinite presentation for $\displaystyle {\bf G}(\mathbf{F}_{q^e}[[t]]) = {\widehat{{\bf G}(A)} \over \{ 1 \} \times \langle\!\langle g \rangle\!\rangle_M}$ which has $d$ generators and $r+1$ relations.
\end{proof}

\medskip

\noindent {\it Proof of Theorem \ref{th::presentation_max_compact}}.
This statement is now a  consequence of the combination of Theorem~\ref{theorem::C} and Proposition~\ref{prop::AL}. \hfill$\square$

\medskip
\begin{remark}
In Theorem 0.1 we assume that $l\geqslant 2$.
We believe it is also true for $l=1$, i.e., for ${\rm SL}_2(\mathbf{F}_q[[t]])$.
\end{remark}

\section{Sylow pro-$p$-subgroups of Chevalley groups over $\mathbf{F}_q(\!(t)\!)$}
\label{s - Sylow}

The other type of compact subgroups we consider in this paper is given by the maximal pro-$p$ subgroups of ${\bf G}\bigl( \mathbf{F}_q(\!(t)\!) \bigr)$.
To prove Theorem 0.2 we will use again Kac-Moody theory, but in a different way than in Theorem  \ref{th::presentation_max_compact}.

The idea is to see the pro-$p$ group we are interested in as the full pro-$p$ completion of a suitable arithmetic group, which can itself be seen as a subgroup of a Kac-Moody group over a finite field. 
Again we use the fact that a presentation for a discrete group naturally leads to a profinite or pro-$p$ presentation of the corresponding completion.


\subsection{The pro-$p$ completions}
\label{ss - pro-p completion}
Let $q=p^a$ $(a\geqslant 1)$ and ${\bf G}$ be a simple simply connected Chevalley group scheme of rank $l \geqslant 3$.
Take $A=\mathbf{F}_q[t]$ and consider $\Gamma_0:={\bf G}(A)$. 
Suppose further that $I=(f(t))$ is an ideal of $A$ generated by an irreducible polynomial $f(t)$ and  $B$ is the completion of $A$ with respect to $I$.
Let $P$ be a pro-$p$-Sylow subgroup of ${\bf G}(B)$.
Then $P$ is an open maximal  pro-$p$ subgroup of ${\bf G}(B)$.
Now consider a very special case: $f(t)=t$, so $B \,\cong\, {\bf F}_q[[t]]$.
Set $\Gamma = \Gamma_0\cap P={\bf G}(A)\cap P$.

\begin{lemma}
\label{key_lemma}
The group $P$ is the pro-$p$ completion of $\Gamma$.
\end{lemma}

\begin{proof}
By the affirmative solution of the CSP (the congruence subgroup problem) for $\Gamma_0$ (see Theorem \ref{th:CSP}), its profinite completion $\widehat{{\bf G}(A)}$ is equal to $\prod_Y  {\bf G}(Y)$, where $Y$ runs over all the completions of $A$ (one of which is $B$). 
As $\Gamma$ is of finite index in $\Gamma_0$, the profinite completion of $\Gamma$ can be easily read from that of $\Gamma_0$:  $\widehat{\Gamma}=\prod_{Y\neq B}{\bf G}(Y)\times P$  when this time $Y$ runs over all of the completions of $A$ except for  $B$.
The completion $B$ contributes the factor $P$.
Now, the pro-$p$ completion of a group is equal (by abstract nonsense) to the maximal pro-$p$ quotient of its profinite completion.
But for $Y\neq B$, the group ${\bf G}(Y)$ has no non-trivial $p$-quotients and so the maximal pro-$p$ quotient of $\widehat{\Gamma}$ is only $P$. 
This finishes the proof.
\end{proof}

\noindent In what follows, we sum up the lemma by writing $P=\Gamma_{\hat{p}}$ (i.e. $\Gamma_{\hat{p}}$ henceforth denotes the pro-$p$ completion of a discrete group $\Gamma$). 

\subsection {Presentations of Sylow $p$-subgroups of ${\rm SL}_3(\mathbf{F}_q)$ and ${\rm Sp}_4(\mathbf{F}_q)$}
\label{ss - unipotent finite} 

As a preparation for proving Theorem 0.2, we will now produce presentations of Sylow $p$-subgroups of  ${\rm SL}_3(\mathbf{F}_q)$  (i.e., $A_2(q)$) and ${\rm Sp}_4(\mathbf{F}_q)$ (i.e., $C_2(q)$), with $q=p^a$. We single out these $p$-groups because, as we will see in the next section, they will turn out to be building blocks of $\Gamma$ from Lemma~\ref{key_lemma}.
\vskip 1mm
$\bf{Notation}:$ For a group $H$ and any $a,b\in H$, denote $[a,b]:=aba^{-1}b^{-1}$ and $a^b:=bab^{-1}$.
\vskip 3mm
In what follows we will need the following well known result attributed to P.~Hall.

\begin{lemma}
\label{lemma::Hall}
Let $G$ be a group that is an extension of $H$ by $N$
$$1\rightarrow N\rightarrow G\rightarrow H\rightarrow 1.$$

Suppose that  $N$ has  a finite presentation $N=\langle n_1, ... , n_r\mid R_1(n_1, ... , n_r), ... , R_k(n_1, ... , n_r)\rangle$
and $H$ has a finite presentation $H=\langle h_1, ... , h_s\mid W_1(h_1, ... , h_s), ... , W_l(h_1, ... , h_s)\rangle$. Then $G$ has the following finite presentation
$$G=\langle n_1, ... , n_r, g_1, ... , g_s \mid R_1(n_1, ... n_r), ... , R_k(n_1, ... , n_r), g_in_jg_i^{-1}=V_{ij}(n_1, ... , n_r),$$
$$g_i^{-1}n_jg_i=U_{ij}(n_1, ... , n_r), 1\leqslant i\leqslant s, 1\leqslant j\leqslant r, W_i(g_1, ... , g_s)=\widetilde{W}_i(n_1, ... , n_r), 1\leqslant i\leqslant l\rangle$$
where $\pi(g_i)=h_i$, $1\leqslant i\leqslant s$, for the natural projection $\pi:G\rightarrow H$ and relations $U_{ij}, V_{ij}$ and $\widetilde{W}_i$ are the obvious suspects.
\end{lemma}

The next statement is Theorem 2 of  \cite{kn::BD}.


\begin{lemma}
\label{lemma::BD}
Let $S$ be a Sylow $p$-subgroup of the finite group ${\rm SL}_3(\mathbf{F}_q)$ where $q=p^a$ for $a\geqslant 1$.
Then  $|S|=q^3=p^{3a}$, the minimal number of generators of $S$ is $2a$, and  $S$ has the following  presentation on $2a$ generators and with $2a(a+1)$ relations.

There are elements $s_1(v_k), s_2(v_{k}) \in S$,  $1\leqslant k\leqslant a$, that generate $S$ and are subject to the following relations:
\begin{description}
\item[(A1)] $(s_i(v_k))^p=1$ for $i=1,2$ and $1\leqslant k\leqslant a$,
\item[(A2)] $[s_i(v_k),s_i(v_{k'})]=1$ for $i=1,2$ and $1\leqslant k< k'\leqslant a$,
\item[(A3)] $[s_1(v_1),[s_1(v_k),s_2(v_1)]]=[s_2(v_1),[s_1(v_k),s_2(v_1)]]=1$ for $1\leqslant k\leqslant a$,
\item[(A4)] $[s_1(v_k)^{-1},s_2(v_{k'})^{-1}]=\prod_{1\leqslant r\leqslant a} [s_1(v_r),s_2(v_1)]^{c(k,k',r)}$ with some fixed \\ 
$c(k,k',1), \ldots, c(k,k',a)\in\mathbf{Z}$, for $1\leqslant k\leqslant a$ and $2\leqslant k' \leqslant a$. 
\end{description}
\end{lemma}
\begin{proof} This is the proof of Theorem 2 of \cite{kn::BD}.
\end{proof}
Remark that for odd $p$,  the presentation of $S$ given above has not only minimal number of generators, but also minimal number of relations \cite[Theorem 3]{kn::BD}.

Notice that under a natural identification of $S$ with the unipotent radical of the standard Borel subgroup of ${\rm SL}_3(\mathbf{F}_q)$ (where the simple roots of $A_2$ are denoted by $\alpha_1$ and  $\alpha_2$), we may choose $s_1(v_k)$ and $s_2(v_{k})$ to correspond to $x_{\alpha_1}(v_k)$ and $x_{\alpha_2}(v_{k})$, $1\leqslant k\leqslant a$, where $v_1, \ldots , v_a$ are chosen to be some generators of $\mathbf{F}_q$, with $v_1=1$.

\subsubsection{Sylow $p$-subgroup of ${\rm Sp}_4(\mathbf{F}_q)$}

We will now discuss presentations of Sylow $p$-subgroups of ${\rm Sp}_4(\mathbf{F}_q)$, $q=p^a$, $a\geq 1$.
For our purposes either $p$ is odd, or  $p=2$ and $q\geq 16$.
Let $S$ be a Sylow $p$-subgroup of ${\rm Sp}_4(\mathbf{F}_q)$.

On the one hand, as ${\rm Sp}_4(\mathbf{F}_q)$ is identified with the universal version of the group $C_2(q)$, it has a natural Steinberg presentation \cite[Theorem 1.12.1]{kn::GLS3} and thus so does $S$.
Unfortunately, this presentation of $S$ has $4(q-1)$ generators and $16(q-1)^2$ relations.

On the other hand, as either $p\geqslant 3$,  or $p=2$ and $q\geq 16$, \cite[Th. 3.3.1]{kn::GLS3} (together with some additional calculations in the case when $p=2$) implies that in fact $S$ has a minimal set of generators  $\{x_{\alpha}(v_i), x_{\beta}(v_i)\mid 1\leqslant i\leqslant a\}$ where $v_1=1$ and $v_i$ ($2\leqslant i\leqslant a$) generate $\mathbf{F}_q$ (here $\alpha$ and $\beta$ are simple roots of $C_2$, with $\alpha$ being a short root and $\beta$ a long one).
In particular,  $S$ is a group of order $q^4=p^{4a}$ with $d(S)=2a$.
We may now apply \cite[Prop. 3.4.1]{kn::LSe} to conclude that $S$ has a presentation with 
$2a$ generators and $8a^2$ relations.

However shorter presentations of $S$ exist and would give more precise estimates for us. 
The reader who is not concerned with this difference is welcome to skip the rest of this subsection.
We will now produce a presentation of $S$ (in the case when $p$ is odd) on $2a$ generators and with $\frac{7a^2+13a}{2}$ relations.
Using this technique, one can also produce a shorter presentation in the case when $p=2$.  However, the calculation is lengthy, and we decided
to demonstrate how it can be done (and get a good estimate) only for $p\geq 3$.
The difference in those two cases happens because of the redundancies in the commutator relations for $p=2$. 

\begin{lemma}
\label{lemma::Sp_4}
Let $S$ be a Sylow $p$-subgroup of the finite group ${\rm Sp}_4(\mathbf{F}_q)$ where $q=p^a$ for $a\geqslant 1$ and $p$ is odd.
Then  $|S|=q^4=p^{4a}$, the minimal number of generators of $S$ is $2a$, and  $S$ has the following presentation on $2a$ generators and with $\frac{7a^2+13a}{2}$ relations.

There are elements $x_{\alpha}(v_k), x_{\beta}(v_{k}) \in S$,  $1\leqslant k\leqslant a$, that generate $S$. For $1\leq i\leq a$, set
$$\xab(v_i)=[\xb(v_1),\xa(v_i)]\prod_{k=1}^a[\xa(\frac{1}{2}v_k), [\xb(v_1),\xa(v_1)]]^{r_k}$$
$$\mbox{and}\ \ \xaa(v_i)=[\xa(\frac{1}{2}v_i),\xab(v_1)]$$ 

\noindent where for each $i$, the element $\xa(\frac{1}{2}v_i)=\Pi_{j=1}^a \xa(v_j)^{m(j,i)}$  with some $m(j,i)\in\mathbf{Z}$, $1\leqslant j\leqslant a$.

 Then $S$ has a presentation with the generators  $x_{\alpha}(v_k), x_{\beta}(v_{k}) \in S$,  $1\leqslant k\leqslant a$, and  the following relations:
\begin{description}
\item[(C1)] $(\xa(v_k))^p=1$, $(\xb(v_k))^p=1$ and $(\xab(v_k))^p=1$ for $1\leqslant k\leqslant a$, 
              \item[(C2)] $[\xa(v_k),\xa(v_{k'})]=1$ for $1\leqslant k < k'\leqslant a$,
        \item[(C3)] $[\xb(v_k),\xb(v_{k'})]=1$ for $1\leqslant k < k'\leqslant a$,
        \item[(C4)] $[\xab(v_k),\xab(v_{k'})]=1$ for $1\leqslant k< k'\leqslant a$,
        \item[(C5)] $[\xa(v_1),[\xa(v_k),\xab(v_1)]]=[\xab(v_1),[\xa(v_k),\xab(v_1)]]=1$ for $1\leqslant k\leqslant a$,
        \item[(C6)] $[\xa(v_k)^{-1},\xab(v_{k'})^{-1}]=\prod_{1\leqslant r\leqslant a} [\xa(v_r),\xab(v_1)]^{c(k,k',r)}$ with some fixed \\ $c(k,k',1), \ldots, c(k,k',a)\in\mathbf{Z}$, for $1\leqslant k\leqslant a$ and $2\leqslant k' \leqslant a$. 
\item[(C7)] $[\xab(v_1), \xb(v_i)]=1$ for $1\leqslant i\leqslant a$,
\item[(C8)]  $[\xab(v_i),\xb(v_1)]=1$ for $1\leqslant i\leqslant a$,
\item[(C9)] $\xaa(v_i^2)^{-1}=\Pi_{k=1}^a\xaa(v_k)^{r_k}$ for all $1\leqslant i\leqslant a$ where $r_k\in\mathbf{Z}$,
\item[(C10)]  $[\xb(v_1),\xa(v_i)]=\xab(v_i)\xaa(v_i^2)$ for $1\leqslant i\leqslant a$.
\item[(C11)] $ [\xa(v_j)^{-1},\xb(v_i)^{-1}]=\xab(v_iv_j)\xaa(v_iv_j^2)^{-1}$ for $1\leqslant i,j\leqslant a$,\\ where 
 $\xab(v_iv_j)=\Pi_{k=1}^a\xab(v_k)^{d(i,j,k)}$ and $\xaa(v_iv_j^{2})=\Pi_{k=1}^a\xaa(v_k)^{f(i,j,k)}$ 
with some $d(i,j,k), f(i,j,k)\in\mathbf{Z}$. 

\end{description}
\end{lemma}

\vskip 3mm
The rest of the section deals with the proof of Lemma~\ref{lemma::Sp_4}.
\vskip 1mm
\begin{proof}
Let $S$ be a Sylow $p$-subgroup of ${\rm Sp}_4(\mathbf{F}_q)$, $q=p^a$, $a\geq 1$ and $p$ odd. Recall that  $|S|=q^4=p^{4a}$ and in the Lie theoretic notation, $S$ is a product of four root subgroups $X_{\alpha}$, $X_{\beta}$, $X_{\alpha+\beta}$ and $X_{2\alpha+\beta}$.   Each root subgroup $X_{\gamma}\cong (\mathbf{F}_q, +)$ and  is generated by root elements  $x_{\gamma}(v_i)$ for some $v_i$'s , $1\leq i\leq a$.

First of all  notice that $S$ is a semidirect product of its normal subgroup $S_0=X_{\alpha}X_{\alpha+\beta}X_{2\alpha+\beta}$ of order $q^3$  and a subgroup $S_1=X_{\beta}$ of order $q$.
The former one is generated by $\xa(v_i)$ and $\xab(v_i)$ (for $1\leqslant i\leqslant a$) and is isomorphic to a Sylow $p$-subgroup of ${\rm SL}_3(\mathbf{F}_q)$.
The latter one is a root subgroup of $S$ generated by $\xb(v_i)$ (for $1\leqslant i\leqslant a$).
To present $S$, we are going to use the presentation of $S_0$ from Lemma~\ref{lemma::BD}, then take an obvious presentation of $S_1$ and finally use P.~Hall's lemma (Lemma~\ref{lemma::Hall} above) to ``glue" those two presentations together in order to obtain a presentation of $S$.
We will then add few relations (that hold in $S$) and show that some of the relations (in this newly obtained presentation) are redundant and can be obtained as a consequence of other relations. 

We  begin by recording a combined list of generators and relations for a presentation of $S_0$ (given in Lemma~\ref{lemma::BD}) and for a presentation of $S_1\cong (\mathbf{F}_q,+)$:
\begin{description}
\item{$\bf{Generators}$:}
$\xa(v_k)$,  $\xab(v_k)$ and $\xb(v_k)$ with $1\leqslant k\leqslant a$.
\item{$\bf{Relations}$:}
        \begin{enumerate}
                \item For $1\leqslant k\leqslant a$, 
                \begin{enumerate} 
                \item $(\xa(v_k))^p=1$,  $(\xb(v_k))^p=1$, and
                \item $(\xab(v_k))^p=1$,
                \end{enumerate} 
        \item $[\xa(v_k),\xa(v_{k'})]=1$ for $1\leqslant k < k'\leqslant a$,
        \item $[\xb(v_k),\xb(v_{k'})]=1$ for $1\leqslant k < k'\leqslant a$,
        \item $[\xab(v_k),\xab(v_{k'})]=1$ for $1\leqslant k< k'\leqslant a$,
        \item $[\xa(v_1),[\xa(v_k),\xab(v_1)]]=[\xab(v_1),[\xa(v_k),\xab(v_1)]]=1$ for $1\leqslant k\leqslant a$,
        \item $[\xa(v_k)^{-1},\xab(v_{k'})^{-1}]=\prod_{1\leqslant r\leqslant a} [\xa(v_r),\xab(v_1)]^{c(k,k',r)}$ with some fixed \\ $c(k,k',1), \ldots, c(k,k',a)\in\mathbf{Z}$, for $1\leqslant k\leqslant a$ and $2\leqslant k' \leqslant a$. 
\vskip 1mm
\end{enumerate}

Let us now  introduce the following notation that will make the next step more comfortable. It comes from the Steinberg notation. Let

$$\xaa(v_i)=[\xa(\frac{1}{2}v_i),\xab(v_1)]\ \ \mbox{ for}\ \  1\leqslant i\leqslant a$$ where $\xa(\frac{1}{2}v_i)=\Pi_{j=1}^a \xa(v_j)^{m(j,i)}$ for all $1\leqslant i\leqslant a$ with some $m(j,i)\in\mathbf{Z}$, $1\leqslant j\leqslant a$.

To use P.~Hall's lemma, we  need  the following additional relations to hold:

\begin{description}
\item[(A)] $\xb(v_i)^{-1}\xa(v_j)\xb(v_i)=R'_{ij}$  which is equivalent to\\ $ [\xa(v_j)^{-1},\xb(v_i)^{-1}]=\xab(v_iv_j)\xaa(v_iv_j^2)^{-1}$ for $1\leqslant i,j\leqslant a$,
\item[(B)] $\xb(v_i)\xa(v_j)\xb(v_i)^{-1}=R_{ij}$ which is equivalent to\\ $[\xa(v_j)^{-1},\xb(v_i)]=\xab(v_iv_j)^{-1}\xaa(v_iv_j^2)$ for  $1\leqslant i,j\leqslant a$,
\item[(C)] $ \xb(v_i)\xab(v_j)\xb(v_i)^{-1}=\xab(v_j)$ for $1\leqslant i,j\leqslant a$,
\item[(D)] $ \xb(v_i)^{-1}\xab(v_j)\xb(v_i)=\xab(v_j)$ for  $1\leqslant i,j\leqslant a$.
\end{description}
\end{description}
In (A) and (B) above, $\xab(v_iv_j)=\Pi_{k=1}^a\xab(v_k)^{d(i,j,k)}$ and $\xaa(v_iv_j^{2})=\Pi_{k=1}^a\xaa(v_k)^{f(i,j,k)}$ 
for all $1\leqslant i,j\leqslant a$ with some $d(i,j,k), f(i,j,k)\in\mathbf{Z}$. 

By Lemma~\ref{lemma::Hall}, the group $S$ has a presentation with $3a$ generators (listed above) and in which the relations are given by (1)--(6) and (A)--(D). 

Let us now observe that the following relations also hold in $S$. They are the consequences of relations (C) and (D) and the Steinberg relations in  ${\rm Sp}_4(\mathbf{F}_q)$ (cf. \cite[Th 1.12.1]{kn::GLS3}).

\begin{description}
\item[$7.$] $[\xab(v_1), \xb(v_i)]=1$ for $1\leqslant i\leqslant a$,
\item[$8.$]  $[\xab(v_i),\xb(v_1)]=1$ for $1\leqslant i\leqslant a$,
\item[$9.$] $\xaa(v_i^2)^{-1}=\Pi_{k=1}^a\xaa(v_k)^{r(k,i)}$ for all $1\leqslant i\leqslant a$ where $r(k,i)\in\mathbf{Z}$,
\item[$10.$]  $[\xb(v_1),\xa(v_i)]=\xab(v_i)\xaa(v_i^2)$ for $1\leqslant i\leqslant a$.
\end{description}

Consider the combined list of relations (1)--(10) and (A), (B), (C), (D). They all hold in $S$ and give a presentation of $S$ on the set of $3a$ generators given above.
We will now show that in fact  we may reduce the list of generators and that relations (B), (C) and (D) follow from (1)--(10) and (A).

In what will follow the following commutator identity will be very useful:
$$[a,bc]=[a,b][a,c]^b\ \ (*)$$

First of all, notice that  (9) and (10) together with the identity $(*)$  imply that for $1\leqslant j\leqslant a$, we have: 

\smallskip

\centerline{$\xab(v_i)=[\xb(v_1),\xa(v_i)]\prod_{k=1}^a[\xa(\frac{1}{2}v_k), [\xb(v_1),\xa(v_1)]]^{r(k,i)}$}

\smallskip

\noindent where for each $k$, the element $\xa(\frac{1}{2}v_k)$ is expressed in terms of $\xa(v_i)$'s  and $\xb(v_i)$'s, $1\leqslant i\leqslant a$, as above. In particular, we may remove $\xab(v_i)$, $1\leqslant i\leqslant a$, from the list of generators.


Now, let us prove the following statement.

\begin{claim}
\label{claim::1}
The following relations follow from the relations (1)--(10) and $(A)$:
\begin{enumerate}
\item $[\xb(v_j),\xaa(v_i)]=1$ for all $1\leqslant i,j\leqslant a$.
\item $[[\xa(v_i),\xb(v_1)],\xb(v_j)]=1$ for $1\leqslant i,j\leqslant a$.
\item $[\xab(v_i),\xb(v_j)]=1$ for $1\leqslant i,j\leqslant a$, i.e.,  relations $(C)$ and $(D)$.
\item  relations $(B)$.
\end{enumerate}
\end{claim}
\begin{proof}
1.Observe that  $[\xaa(v_i),\xb(v_j)]=[[\xa(\frac{1}{2}v_i),\xab(v_1)],\xb(v_j)]$ for $1\leqslant i,j,\leqslant a$.
Recall now  the Hall-Witt identity formulated for our definition of the commutators:
$$[[y,x^{-1}],z^{-1}]^{y^{-1}}[[z,y^{-1}],x^{-1}]^{z^{-1}}[[x,z^{-1}],y^{-1}]^{x^{-1}}=1$$
Let us apply  it with $y=\xa(\frac{1}{2}v_i)$, $x=\xab(v_1)^{-1}$ and $z=\xb(v_j)^{-1}$.
First notice that $[[x,z^{-1}],y^{-1}]=[[\xab(v_1)^{-1},\xb(v_j)],\xa(\frac{1}{2}v_i)^{-1}]=[1, \xa(\frac{1}{2}v_i)^{-1}]=1$  because of  (7).
Now, $[[z,y^{-1}],x^{-1}]=[[\xb(v_j)^{-1},\xa(\frac{1}{2}v_i)^{-1}],\xab(v_1)]=[[\xb(v_j)^{-1},\Pi_{k=1}^a \xa(v_k)^{-m(i,k)}],\xab(v_1)]$ by (2).
Let us look more closely at $[\xb(v_j)^{-1},\Pi_{k=1}^a \xa(v_k)^{-m(i,k)}]$. Using repeatedly the commutator identity $(*)$ together with relation $(A)$ and the fact that the subgroup $H_0:=\langle \xab(v_i), \xaa(v_i), 1\leqslant i\leqslant a\rangle$ of $S_0$  is normal in $S_0$, we see that $[\xb(v_j)^{-1},\xa(\frac{1}{2}v_i)^{-1}]\in H_0$ for all $1\leqslant i, j\leqslant a$. Now the structure of  $S_0$ tells us that $H_0$ is abelian and as $\xab(v_1)\in H_0$, we conclude that $[[\xb(v_j)^{-1},\xa(\frac{1}{2}v_i)^{-1}],\xab(v_1)]=1$.  \footnote{We may use the facts about the  structure of $S_0$ since $S_0$ is given by relations (1), (2), (4), (5) and (6), and (9) holds in $S_0$.}
At last the Hall-Witt identity gives us that $[[\xa(\frac{1}{2}v_i),\xab(v_1)],\xb(v_j)]=1$ for $1\leqslant i,j\leqslant a$, which implies the desired result.
\vskip 1mm

2. Our proof follows a similar one in \cite{kn::BD}. We have that\\
$\xb(v_j)[\xa(v_i),\xb(v_1)]= \xa(v_i)\xb(v_j)\xb(v_j)^{-1}\xa(v_i)^{-1}\xb(v_j)\xa(v_i)\xb(v_1)\xa(v_i)^{-1}\xb(v_1)^{-1}=$
$=\xa(v_i)\xb(v_j)[\xb(v_j)^{-1},\xa(v_i)^{-1}]\xb(v_1)\xa(v_i)^{-1}\xb(v_1)^{-1}$.

Since $[\xb(v_j)^{-1},\xa(v_i)^{-1}]=[\xa(v_i)^{-1},\xb(v_j)^{-1}]^{-1}$, using relations $(A)$ together with part 1 and relations (8),
we observe that $[\xb(v_j)^{-1},\xa(v_i)^{-1}]\xb(v_1)=\xb(v_1)[\xb(v_j)^{-1},\xa(v_i)^{-1}]$.  As furthermore, $\xb(v_1)$ commutes with $\xb(v_j)$, we have that\\
$\xb(v_j)[\xa(v_i),\xb(v_1)]=\xa(v_i)\xb(v_1)\xb(v_j)\xb(v_j)^{-1}\xa(v_i)^{-1}\xb(v_j)\xa(v_i)\xa(v_i)^{-1}\xb(v_1)^{-1}=$
$\xa(v_i)\xb(v_1)\xa(v_i)^{-1}\xb(v_j)\xb(v_1)^{-1}=\xa(v_i)\xb(v_1)\xa(v_i)^{-1}\xb(v_1)^{-1}\xb(v_j)=[\xa(v_i),\xb(v_1)]\xb(v_j)$.

\vskip 1mm

3. From (10) and (9) we have   $$\xab(v_i)=[\xb(v_1),\xa(v_i)]\xaa(v_i^2)^{-1}=[\xb(v_1),\xa(v_i)]\Pi_{k=1}^a\xaa(v_k)^{r(k,i)}$$
As $[a,bc]=[a,b][a,c]^b$,
$[\xb(v_j), \xab(v_i)]=[\xb(v_j), [\xb(v_1),\xa(v_i)]\Pi_{k=1}^a\xaa(v_k)^{r(k,i)}]=$\\
$=[\xb(v_j), [\xb(v_1),\xa(v_i)]][\xb(v_j),\Pi_{k=1}^a\xaa(v_k)^{r(k,i)}]^{[\xb(v_1),\xa(v_i)]}$.

Using part 2, we see that  $[\xb(v_j), [\xb(v_1),\xa(v_i)]]=1$, and from part 1 it follows that the commutator $[\xb(v_j),\Pi_{k=1}^a\xaa(v_k)^{r(k,i)}]=1$.  This gives (C).
Now notice that $(D)$ is an immediate consequence of $(C)$.

4. Finally we observe that relations $(A)$ together with parts 1 and 3  imply that
$\xb(v_i)^{-1}$ commutes with $[\xa(v_j)^{-1},\xb(v_i)^{-1}]$.
Therefore by \cite[Lemma 2]{kn::BD} we have:
$$[\xa(v_j)^{-1},\xb(v_i)]=[\xa(v_j)^{-1},\xb(v_i)^{-1}]^{-1},$$
which in turn (using the relations in $S_0$) implies relations $(B)$. 
\end{proof}

Thus we have shown that $S$ has a presentation with the generators $\xa(v_k)$ and $\xb(v_k)$, $1\leqslant k\leqslant a$, subject to the relations (1)--(10) and (A).
Rename the relations (1)--(10) into (C1)--(C10) and relations (A) into (C11).
Finally,  counting the number of those relations, we obtain a presentation of $S$ on $2a$ generators and the promised number of relations.
\end{proof}

\subsection{Kac-Moody groups: presentation of $\Gamma=U_+$}

Recall that $\overline{G}={\bf G}\bigl(\mathbf{F}_q(\!(t)\!)\bigr)$ can be thought of as a topological Kac-Moody group.
Since the rank of it as a Chevalley group is $l$, its Kac-Moody rank is $l+1$.
Then ${\bf G}(\mathbf{F}_q[[t]])$ can be naturally identified with a maximal parahoric subgroup and $P$ with the pro-unipotent radical  $\overline{U}_+$ of the
standard Iwahori  subgroup $\overline{B}_+$.
Thus  $P=\overline{U}_+$  is the closure of $U_+$ in $\overline{G}$ where $U_+$ is the subgroup of the corresponding minimal Kac-Moody group $\widetilde{G} = {\bf G}({\bf F}_q[t,t^{-1}])$ generated by the positive real root subgroups of $\widetilde{G}$ \cite[Theorem 1.C]{RemRon}.

In the notation of \S\ref{ss - pro-p completion}, we apply Lemma \ref{key_lemma}~with $\Gamma = U_+$: this tells us that $P$ is the pro-$p$ completion of the group $U_+$ which, in what follows, is better understood than $P$ thanks to a combination of \S\ref{ss - unipotent finite} and of Kac-Moody arguments which we explain in the rest of the subsection.
We will then deduce a presentation of $P$ from a presentation of $\Gamma = U_+$ because $P = \Gamma_{\hat{p}}$.

Let us go first a bit deeper into Kac-Moody theory. 
Let $A$ be the generalized Cartan matrix of $\widetilde{G}$ and $\pi=\{\alpha_0, \alpha_1, \ldots , \alpha_{l}\}$ be the set of its fundamental roots.
Corollary 1.2 of  \cite{kn::DM}  implies that if $A$ is $3$-spherical and $q\geqslant 16$, then $U_+$ is an amalgamated product of the system $\{X_{\alpha_i}\}\cup\{X_{\alpha_i,\alpha_j}\}$, $0\leqslant i,j\leqslant l$, where each $X_{\alpha_i}=\langle x_{\alpha_i}(c)\mid c\in \mathbf{F}_q\rangle\cong (\mathbf{F}_q,+)$ is a fundamental root subgroup of $\widetilde{G}$ and $X_{\alpha_i,\alpha_j}=\langle X_{\alpha_i}, X_{\alpha_j}\mid R_{i,j}\rangle$ where $R_{i,j}$ comes from the rank $2$ subsystem with fundamental roots  $\alpha_i$ and $\alpha_j$.
To simplify things, let us denote $X_i:=X_{\alpha_i}$ and $X_{i,j}:=X_{\alpha_i,\alpha_j}$  (in particular, remark that $X_{i,j}=X_{j,i}$ for all $0 \leqslant i,j \leqslant l$).
Thus in fact, $U_+$ is a group generated by the elements of the root subgroups $X_i$'s for $i=0,1, ..., l$ and presented by  the relations between those generators that hold in $X_{i,j}$ for $0\leqslant i,j\leqslant l$ (we may let $X_{i,i}:=X_i$).

Recall that the $3$-spherical condition is relevant to generalized Dynkin diagrams as obtained for instance in Kac-Moody theory. 
In our situation the Dynkin diagram of the (untwisted) affine Kac-Moody group $\widetilde{G}={\bf G}({\bf F}_q[t,t^{-1}])$ obtained from ${\bf G}$, is the diagram obtained by adding suitably a single vertex to the (classical, i.e. spherical) Dynkin diagram of ${\bf G}$.
The resulting Dynkin diagram is called the completed Dynkin diagram of ${\bf G}$; the vertex and the edges emanating from it are added in such a way that any subdiagram obtained by removing an arbitrary vertex (and the edges emanating from it) is of finite type. 
Geometrically, this corresponds to the fact that we pass from a finite Weyl group to an affine one acting on a Euclidean tiling so that any vertex stabilizer is a finite Weyl group. 
In other words, if ${\bf G}$ is a Chevalley group of rank $l$, then the Kac-Moody group $\widetilde{G}$ is $l$-spherical. 

Let us now discuss the tools we will need to obtain  an explicit presentation of $U_+$.
Let $v_1, \ldots, v_a$ be generators of $\mathbf{F}_q$  with $v_1=1$.
It is obvious that for each $i$,

$$\langle \{ x_{\alpha_i}(v_k) \}_{1\leqslant k\leqslant a}\mid x_{\alpha_i}(v_k)^p=1, [x_{\alpha_i}(v_k), x_{\alpha_i}(v_{k'})]=1 \, \hbox{\rm for} \, 1\leqslant k < k' \leqslant a\rangle$$ 
is a presentation of $X_i$.
 Therefore $U_+$ is generated by elements $\{ x_{\alpha_i}(v_k)\}_{1\leqslant k\leqslant a, 0\leqslant i\leqslant l}$ and the following 
 $\displaystyle a(l+1)+\frac{1}{2}a(a-1)(l+1)=\frac{a(a+1)(l+1)}{2}$
 relations must hold:
 
 \begin{enumerate}
 \item $x_{\alpha_i}(v_k)^p=1$ for  $1\leqslant k\leqslant a$ and $0\leqslant i\leqslant l$.
 \item  $[x_{\alpha_i}(v_k), x_{\alpha_i}(v_{k'})]=1$ for $1\leqslant k < k' \leqslant a$ and $0\leqslant i\leqslant l$.
 \vskip 3mm
 The remaining relations will come from the subsystems $X_{i,j}$ for $0\leqslant i\neq j\leqslant l$.
 Clearly, those depend on the type of $X_{i,j}$ which is determined by the type of root system generated by $\alpha_i$ and $\alpha_j$.
 Let us discuss those case-by-case.
 
 Suppose first that $X_{i,j}$ is of type $A_1\times A_1$.
 Then $X_{i,j}\cong X_i\times X_j$ and  we would need the following additional relations  to describe $X_{i,j}$ in this case:
 
 \item $[x_i(v_k), x_j(v_{k'})]=1$ for $1\leqslant k,k' \leqslant a$.
 Thus for every subsystem $X_{i,j}$ of type $A_1\times A_1$, we will have $a^2$ additional relations.
 \vskip 3mm
\item  If $X_{i,j}$ is of type $A_2$, then $X_{i,j}$ is isomorphic to a Sylow $p$-subgroup of ${\rm SL}_3(\mathbf{F}_q)$, and so using Lemma~\ref{lemma::BD}, we notice that we need relations (A3) and (A4) with $s_i(v_k)=x_i(v_k)$, $1\leq k\leq a$.
Thus for every subsystem $X_{i,j}$ of type $A_2$, we will have $2a+a(a-1)=a(a+1)$ additional relations.
\vskip 3mm
\item Finally, if $X_{i,j}$ is of type $C_2$, then $X_{i,j}$ is isomorphic to a Sylow $p$-subgroup of ${\rm Sp}_4(\mathbf{F}_q)$. If $p$ is odd, we may use the results of
\S 2.2.1 telling us that we need $$\frac{7a^2+13a}{2}-2a-a(a-1)=\frac{5a^2+11a}{2}$$ additional relations. These are relations (C4)--(C11) and some of relations (C1).  
If $p=2$, we use the result of  \cite[Prop. 3.4.1]{kn::LSe} to say that $X_{i,j}$ requires $8a^2$ additional relations.
\end{enumerate}

 Finally, notice that all of the above relations are products of $p$-powers and commutators.
 Therefore, $U_+/([U_+,U_+]U_+^p)$ is an elementary abelian $p$-group of order at most $p^{a(l+1)}$.
 To see that the exact order of this quotient is $p^{a(l+1)}$, it is enough to exhibit a quotient of $U_+$ which is an elementary abelian $p$-group whose order is equal to this upper bound. 
 This can be seen by considering the action of $U_+$ on the set of alcoves sharing a codimension 1 face with the alcove $c$ stabilized by $U_+$. 
 These alcoves are in one-to-one correspondence with the set of simple roots of the Kac-Moody group $\widetilde{G}$ defined by $c$. 
 Moreover, by the commutation relations between root groups in $\widetilde{G}$, each given simple root group acts simply transitively on the alcoves sharing a given panel with $c$ (but $\neq c$) and acts trivially on the remaining alcoves of the finite set under consideration. 
 The image of $U_+$ for this action is the desired quotient. 
 Therefore we  obtain that $$U_+/([U_+,U_+]U_+^p)\cong \mathbf{F}_p^{a(l+1)}$$
 (see also Corollary 2.5 of \cite{kn::CR}).
 
 We can now summarise:
 
\begin{prop}
\label{prop::presentation_of_Gamma}  
Let $\widetilde{G} = {\bf G}({\bf F}_q[t,t^{-1}])$ be a minimal affine untwisted Kac-Moody group of rank $l+1\geq 4$ defined over a field $\mathbf{F}_q$ where $q=p^a$, $a\in\mathbf{N}$, $q\geqslant 16$. Let  $U_+$ be  the subgroup of $\widetilde{G}$ generated by the positive real root subgroups. Then the  
the following conditions hold:
\begin{enumerate}
\item$U_+/([U_+,U_+]U_+^p)\cong \mathbf{F}_p^{a(l+1)}$, and so  $d(U_+)=a(l+1)$.
\item The group $U_+$ has a presentation $\langle \{ x_i(v_k) \}_{0\leqslant i\leqslant l, 1\leqslant k\leqslant a} \mid R \rangle$ such\ that
 $R$ are as described in (1)--(5) above.
\item $$r(U_+)\leqslant \frac{a(a+1)(l+1)}{2}+ |\{X_{i,j}\mid  X_{i,j} \ \mbox{is\ of\ type}\ A_1\times A_1\}| \cdot a^2$$ 
$$+ \, |\{X_{i,j}\mid  X_{i,j} \ \mbox{is\ of\ type}\ A_2\}| \cdot a(a+1) + |\{X_{i,j}\mid  X_{i,j} \ \mbox{is\ of\ type}\ C_2\}| \cdot |R_{C_2}|$$ where 
$|R_{C_2}|=\frac{5a^2+11a}{2}$ if $q$ is odd, and $|R_{C_2}|=8a^2$ if $q$ is even.
\end{enumerate}
%
%
\end{prop}
 
%


We may now evaluate $r(U_+)$ in each case using Proposition~\ref{prop::presentation_of_Gamma}. 
We then record the datum and the outcomes of the calculations  in Table 1. For groups of type $B_n$, $C_n$ and $F_4$, we only record the estimates for $p$ odd (though all the estimates appear in the next section). We do it because the bound for $p=2$ is quite far away from being sharp.

\begin{table}
\label{table1}
\begin{center}
\caption{}
\vskip 5mm
\begin{tabular}{|c|c|c|c|c|c|} \hline
Type of           & Kac-Moody  diagram of    & $|\{X_{i,j}\}|$ for                   &  $|\{X_{i,j}\}|$ for &  $|\{X_{i,j}\}|$ for  & $r(U_+)\leqslant$\\ 
$\mathbf{G}$ &  $\widetilde{\mathbf{G}}$  &    $X_{i,j}=A_1\times A_1$   &   $X_{i.j}=A_2$  &  $X_{i,j}=C_2$     &($^*$ - is given only for $p\geq 3$)           \\  \hline

&&&&& \\

$A_l$ & 
\includegraphics[width=3cm]{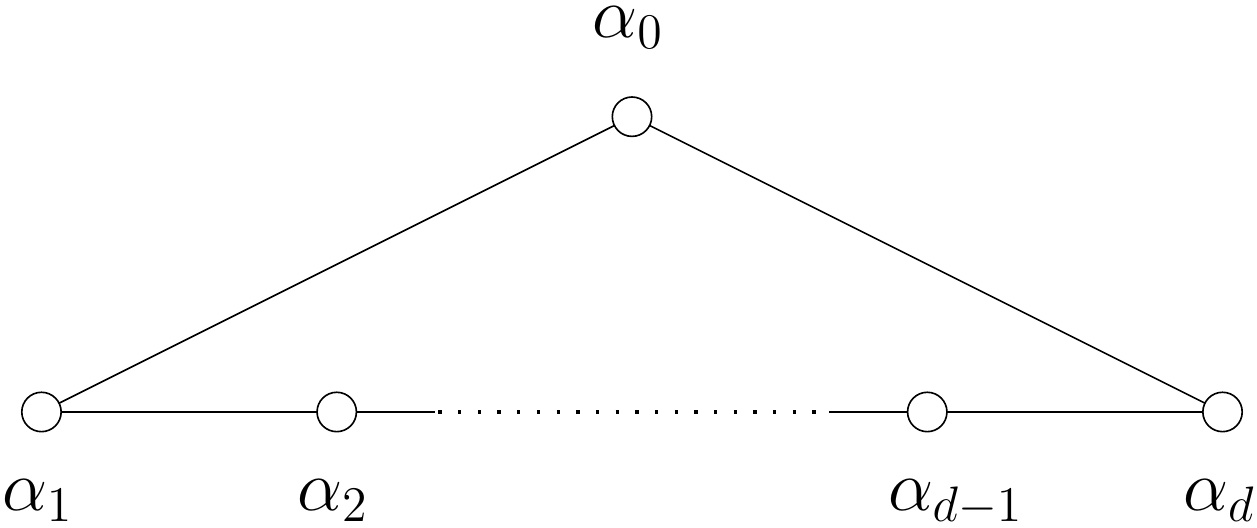}
  &  $\frac{(l+1)(l-2)}{2}$  & $l+1$  & $0$   & $\frac{a^2(l+1)^2+3a(l+1)}{2}$  \\

&&&&& \\ \hline

&&&&& \\

$B_l$ & 
\includegraphics[width=3cm]{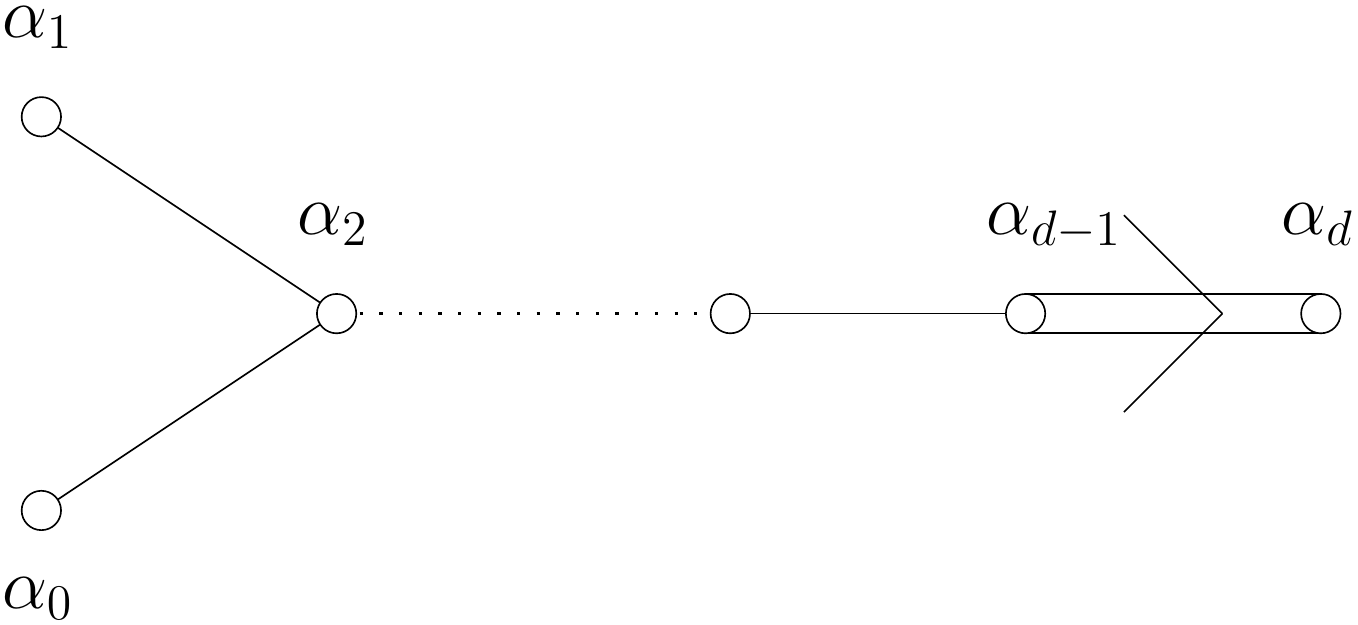}
  &  $\frac{l(l-1)}{2}$  & $l-1$  & $1$   & $\frac{a^2(l+1)^2+3a(l+1)+3a^2+7a}{2}^*$  \\

&&&&& \\ \hline

&&&&& \\

$C_l$ & 
\includegraphics[width=3cm]{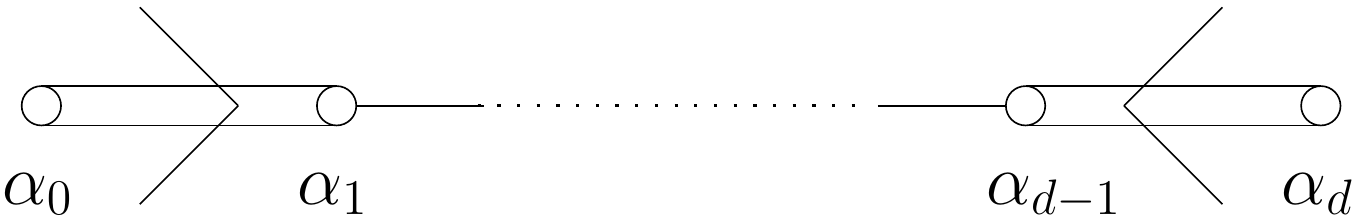}
  &  $\frac{l(l-1)}{2}$  & $l-2$  & $2$   & $\frac{a^2(l+1)^2+3a(l+1)+6a^2+16a}{2}^*$  \\

&&&&& \\ \hline

&&&&& \\

$D_l$ & 
\includegraphics[width=3cm]{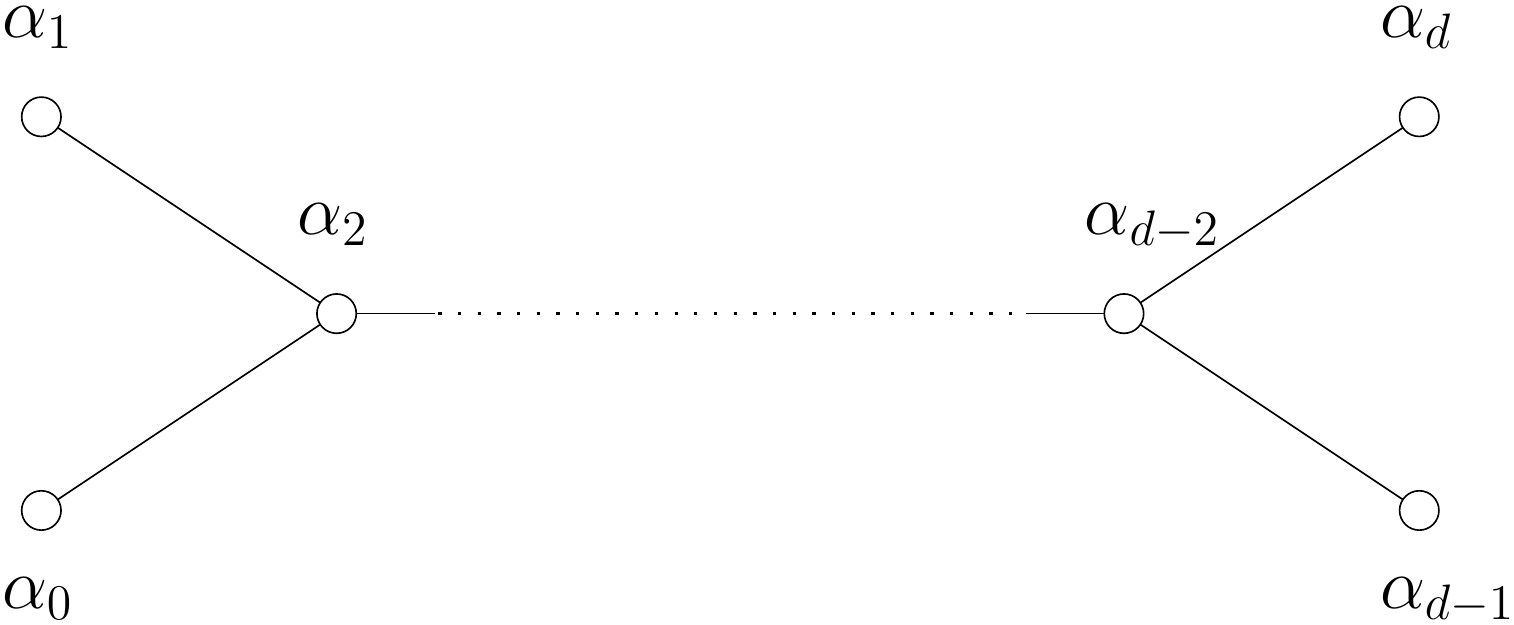}
  &  $\frac{l(l-1)}{2}$  & $l$  & $0$   & $\frac{a^2(l+1)^2+3a(l+1)-2a}{2}$  \\

&&&&& \\ \hline

&&&&& \\

$E_l$ & 
\includegraphics[width=3cm]{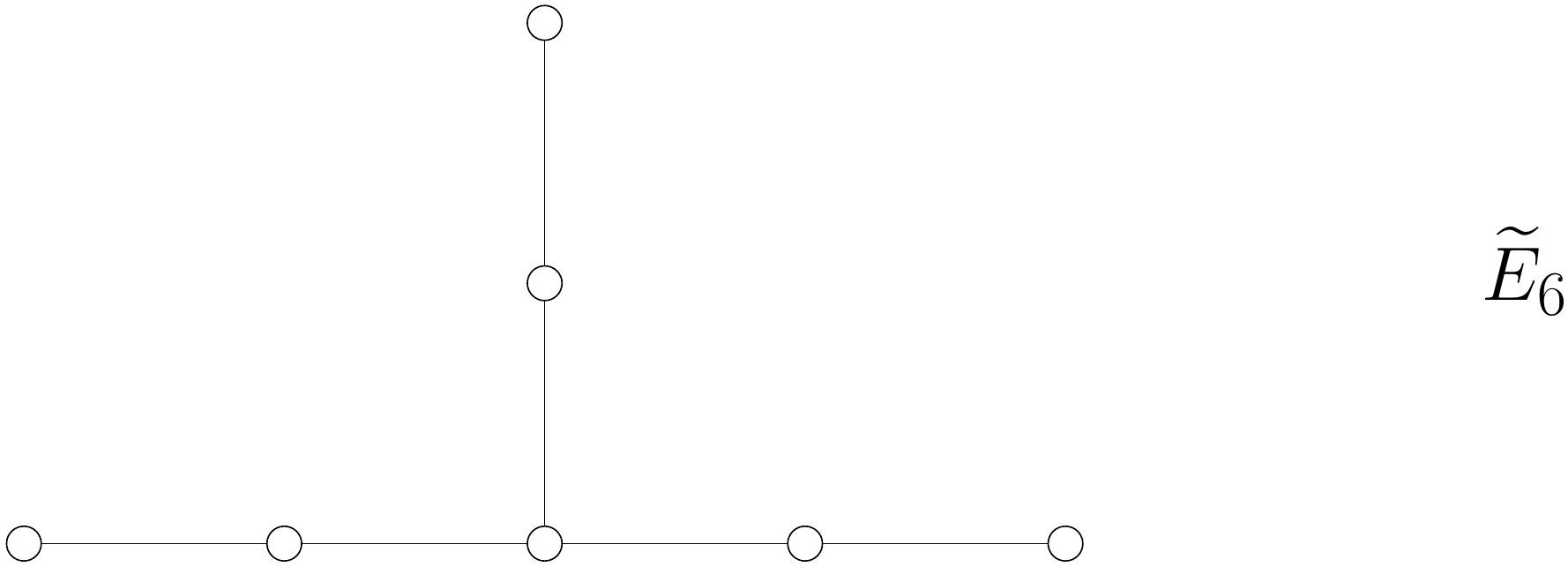}
  &  $\frac{l(l-1)}{2}$  & $l$  & $0$   & $\frac{a^2(l+1)^2+3a(l+1)-2a}{2}$  \\

&&&&& \\ 
&\includegraphics[width=3cm]{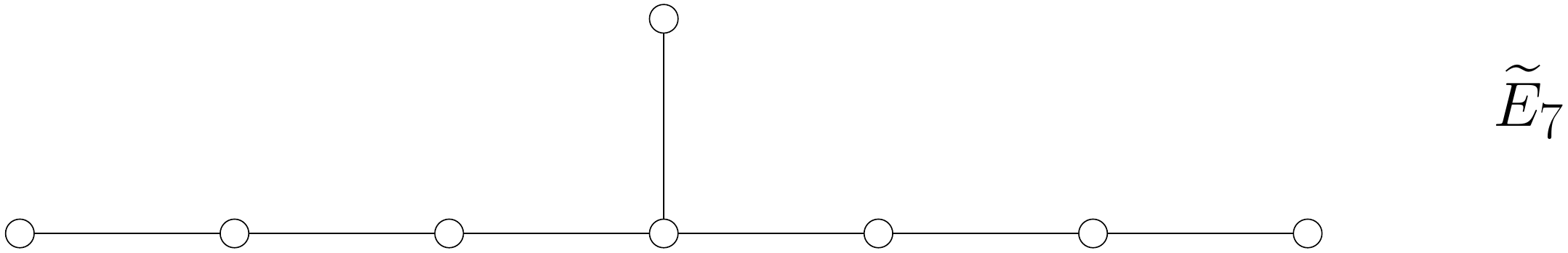}&&&& \\ 
&&&&& \\ 
&\includegraphics[width=3cm]{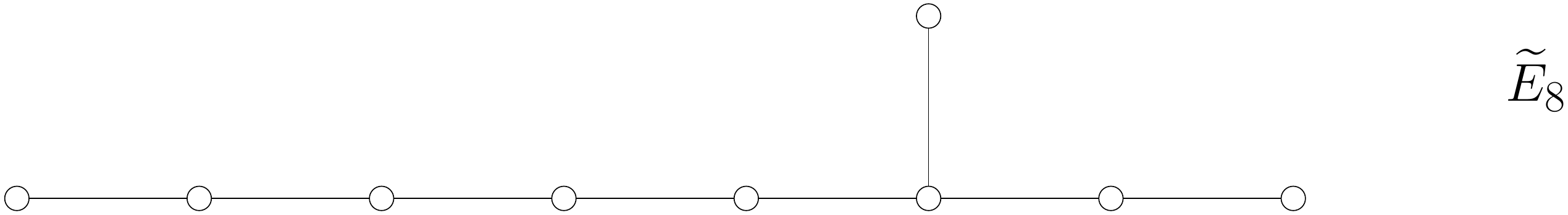}&&&& \\ 

&&&&& \\ \hline

&&&&& \\

$F_4$ & 
\includegraphics[width=3cm]{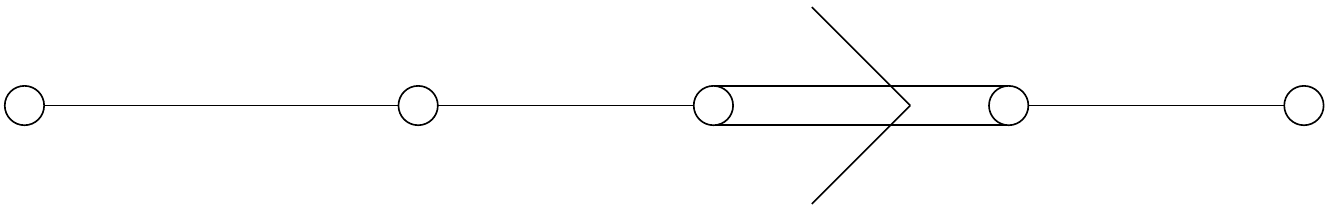}
  &  $6=\frac{4\cdot 3}{2}$  & $3=4-1$  & $1$   & $14a^2+11a \ \ ^*$  \\

&&&&& \\ \hline

\end{tabular}
\end{center}
\end{table}

\subsection{Conclusion}
\label{ss - ccl Sylow}

The number of generators of $P=\Gamma_{\hat{p}}$ in its  presentation coming from $\Gamma$ is $a(l+1)$. This is in fact a presentation with a minimal number of generators (by Proposition~\ref{prop::presentation_of_Gamma}  -- see also  \cite[Cor. 2.5]{kn::CR}). Moreover the generators  can be chosen to be  $x_0(v_k), ... , x_{l}(v_k)$, $1\leqslant k\leqslant a$.
Our discussion gives us that $P$ has a presentation $\langle x_0(v_k), ... , x_{l}(v_k), 1\leqslant k\leqslant a \mid R\rangle$ with $|R| \, \leqslant Ca^2(l+1)^2$ (with $C$ an appropriate constant as calculated in the previous section). 
More precisely, this proves the upper bound in the next theorem: 

\begin{theorem}
\label{th - precise presentation pro-p-Sylow}
Let ${\bf G}\bigl(\mathbf{F}_q(\!(t)\!)\bigr)$ be a simple, simply connected, Chevalley group of rank $l$ over a field $\mathbf{F}_q(\!(t)\!)$ with $q=p^a\geqslant 16$, and let $P$ be its Sylow pro-$p$-subgroup.
Then 
$$\frac{a^2(l+1)^2}{4}\leqslant r(P)$$
and the following upper bounds hold:
\begin{enumerate}
\item[{\rm 1.}]~If $G$ has type $A_{l}$ and $l\geqslant  3$, we have $r(P)\leqslant \frac{a^2(l+1)^2+3a(l+1)}{2}$.

\item[{\rm 2.}]~For type $B_{l}$ and $l\geqslant 3$, we have $r(P)\leqslant \frac{a^2(l+1)^2+3a(l+1)+3a^2+7a}{2}$ for $p\geqslant 3$,\\
and $r(P)\leqslant \frac{a^2(l+1)^2+3a(l+1)+14a^2-4a}{2}$ for $p=2$.

\item[{\rm 3.}]~For type $C_{l}$ and  $l\geqslant 3$, we have $r(P)\leqslant \frac{a^2(l+1)^2+3a(l+1)+6a^2+16a}{2}$ and $p\geqslant 3$,\\
and $r(P)\leqslant \frac{a^2(l+1)^2+3a(l+1)+28a^2-6a}{2}$ for $p=2$.

\item[{\rm 4.}]~For type $D_{l}$ and $l\geqslant 4$, we have $r(P)\leqslant \frac{a^2(l+1)^2+3a(l+1)-2a}{2}$.

\item[{\rm 5.}]~For type $E_{l}$ and $l\in\{6,7,8\}$, we have $r(P)\leqslant \frac{a^2(l+1)^2+3a(l+1)-2a}{2}$.

\item[{\rm 6.}]~For type $F_4$, we have $r(P)\leqslant 14a^2+11a$ for $p\geqslant 3$, and $r(P)\leqslant 15a^2+4a$ for $p=2$.

.
\end{enumerate} 
\end{theorem}

\begin{proof}
It remains to prove the lower bound. 
For this we use the Golod-Shafarevich inequalities as stated in the Propostion below, with $Q=P$ and $d(P)=a(l+1)$.
\end{proof}

\begin{prop}
\label{GS}
Let $Q$ be a pro-$p$ subgroup of ${\bf G}\bigl( {\bf F}_q (\!( t )\!) \bigr)$ as in Theorem \ref{th - precise presentation pro-p-Sylow}. 
Then, with the obvious definitions for $d(Q)$ and $r(Q)$, we have $\displaystyle r(Q) \geqslant {d(Q)^2 \over 4}$. 
\end{prop}

\begin{proof}
We follow the lines given by the proof of Prop. 5.4 and the discussion after Prop. 5.5 in \cite{kn::LS}. 
With the notation and terminology there, the group $Q$ is commensurable with an ${\bf F}_q[[t]]$-standard group [loc. cit., Def. 2.1], which implies the subexponential growth of the sequence $(r_n)_{n \geqslant 0}$ related to the powers of the augmentation ideal $\Delta$ of the group ring ${\bf F}_qQ$ (more precisely: $r_n$ is defined to be ${\rm dim}_{{\bf F}_p}(\Delta^n/\Delta^{n+1})$ [loc. cit., p. 320]).
Subexponential growth of $(r_n)_{n \geqslant 0}$ implies the desired Golod-Shafarevich inequality.
\end{proof}



The equality $P=\Gamma_{\hat{p}}=(U_+)_{\hat{p}}$ was used  to derive an upper bound on the number of relations needed to present $P$ as a pro-$p$ group.
But it can be also used to deduce a lower bound on the number of relations needed to present $\Gamma=U_+$ in the discrete category.

\begin{cor}
\label{cor::2.5}
For $\Gamma=U_+$ we have:

$$d(\Gamma)=a(l+1)\ \ \ \mbox{and}\ \ \ \frac{a^2(l+1)^2}{4}\leqslant r(\Gamma)\leqslant 6.25\cdot\frac{a^2(l+1)^2}{4}.$$

\noindent 


\end{cor}

\begin{proof}
The result on $d(\Gamma)$ was shown above (Proposition~\ref{prop::presentation_of_Gamma}) and so was the upper bound on $r(\Gamma)$
(in fact, slightly better bounds are given in Table 1).
The lower bound is deduced from the fact that $r(\Gamma)\geqslant r_p(\Gamma_{\hat{p}})=r_p(P)$ and the Golod-Shafarevich inequality given in Proposition \ref{GS}.
\end{proof}

\begin{remark}
The last corollary is slightly surprising: usually the Steinberg type presentations of groups are far from being optimal. 
In our case the Steinberg-like presentations are essentially optimal. 
\end{remark}

\section{The characteristic $0$ case}

Some results like Theorems \ref{th::presentation_max_compact} and 0.2 can be proved also in characteristic $0$. But our results in this case are weaker than the above theorems.

In this section let $k$ be a number field, $\mathcal{O}$ its ring of integers and $V$ the set of its   finite  valuations. 
For $v\in V$, let $k_v$ be the completion of $k$ with respect to $v$, and $\mathcal{O}_v$ the closure of $\mathcal{O}$ in $k_v$.
For such a $v$, denote by $\mathbf{F}_{p_v^{a_v}}$ the finite residue field of $k$.
Let $\bf{G}$   be a simple simply connected Chevalley group scheme of arbitrary rank $l\geqslant 1$, let $G={\bf G}(\mathcal{O}_v)$ and let $P$ be a Sylow pro-$p_v$ subgroup of $G$. Here is what we can prove in this situation.

\begin{theorem}
\label{theorem::3.1}
There exists a constant $C_1=C_1(k)$ such that for every $v\in V$ and every ${\bf G}$  as above, $G={\bf G}(\mathcal{O}_v)$ has a presentation $\Sigma(G)$, with $D_{\Sigma(G)}$ generators and $R_{\Sigma(G)}$ relations, satisfying 
$$D_{\Sigma(G)}+R_{\Sigma(G)}\leqslant C_1.$$
\end{theorem}

\vskip 3mm

\begin{theorem}
\label{theorem::3.2}
With the notations as above, assume that  the following conditions hold:
\begin{enumerate}
         \item If $l=1$, then $p_v>2$.
         \item If $l=2$, then  $p_v>2$ if $\mathbf{G}$ is $B_2=C_2$ and that $p_v> 3$ if $\mathbf{G}=G_2$. 
         \item If $l\geqslant 3$, then one of the following conditions hold:
                 \begin{enumerate}
                 \item $p_v^{a_v}\geqslant 16$, or
                 \item if $p_v^{a_v}<16$, then $p_v>2$ if $\mathbf{G}$ is $B_l, C_l$ or $F_4$.
                 \end{enumerate}
\end{enumerate}         
   Then  $P$ is generated by at least $a_v l$ and at most $a_v(l+1)$ generators.
Moreover,  there exist absolute constants $C_2\geqslant 0$ and $C_3\geqslant 0$ such that:
\begin{itemize} 
\item the group $P$ has a presentation with $a_v(l+1)$ generators and at most $C_2a_v^2l^4$ relations;
\item any presentation of $P$ needs at least $C_3a_v^2l^2$ relations.
\end{itemize}
\end{theorem}

\vskip 3mm
Theorem~\ref{theorem::3.1} can be deduced from the following result.
\vskip 3mm
\begin{theorem}
\label{theorem::3.3}
There exists a constant $C_4=C_4(k)$ such that $\Lambda={\bf G}(\mathcal{O})$ has a presentation $\Sigma_{\Lambda}$, with $D_{\Sigma_{\Lambda}}$ generators and $R_{\Sigma_{\Lambda}}$ relations, satisfying 
$$D_{\Sigma_{\Lambda}}+R_{\Sigma_{\Lambda}}\leqslant C_4.$$
\end{theorem}

Let us make the following remark.

\begin{remark}
\label{remark::34}
\begin{enumerate}
\item We believe that the lower bound in Theorem~\ref{theorem::3.2} above is sharp.
\item We do not know if the constant $C_1(k)$ in Theorem~\ref{theorem::3.1} really depends on $k$.
\item On the other hand, $C_4$ does depend on $k$, or at least on the degree $[k:\mathbf{Q}]$.
To see this, let $k_i$ be a sequence of number fields with $[k_i:\mathbf{Q}]=n_i\rightarrow\infty$ such that there exists a rational prime $p$ which splits completely in all the $k_i$'s.
Then by the Strong Approximation Theorem, the group ${\bf G}(\mathcal{O})$ is mapped onto $G(\mathbf{F}_p)^{n_i}$ and since $d({\bf G}(\mathbf{F}_p)^{n_i})$ grows logarithmically with $n_i$ \cite{kn::W}, $d({\bf G}(\mathcal{O}))$ grows to infinity with $[k:\mathbf{Q}]$.

Such a sequence of fields exists.
For example, denote by $\{p_i\}_{i\in\mathbf{N}}$ the sequence of odd primes congruent to $2=3^2$ modulo $7$ : 
$p_1=23$, $p_2=37$, etc. and then set $k_0=\mathbf{Q}$ and then $k_{i+1}=k_i(\sqrt{p_{i+1}})$. 
The polynomial $x^2-p_{i+1}$ is irreducible over  $k_i$, but reducible over $\mathbf{F}_7$. 
Hence, if $\mathcal{O}_i$ denotes the ring of integers of $k_i$, we should have $\mathcal{O}_i/7\mathcal{O}_i = (\mathbf{Z}/7\mathbf{Z})^{2^i}$.
\end{enumerate}
\end{remark}

Let us postpone the proof of Theorem~\ref{theorem::3.3} until  Section 4. Instead now we show how it implies Theorem~\ref{theorem::3.1}.
\vskip 5mm
\emph{Proof of Theorem~\ref{theorem::3.1}}.
Let us first assume that ${\bf G}$ is of rank $\geqslant 2$. 
Then the group ${\bf G}(\mathcal{O})$ has the  congruence subgroup property in the sense that $\widehat{{\bf G}(\mathcal{O})}\rightarrow {\bf G}(\widehat{\mathcal{O}})$
is onto with a cyclic kernel (see Theorem \ref{th:CSP}).
The presentation of ${\bf G}(\mathcal{O})$, promised in Theorem~\ref{theorem::3.3}, gives a profinite presentation of $\widehat{{\bf G}(\mathcal{O})}$
and so, with one additional relation used to kill the cyclic kernel if needed, it serves as a presentation of ${\bf G}(\widehat{\mathcal{O}})$.
As in the proof of Proposition 1.2, since ${\bf G}(\widehat{\mathcal{O}})={\bf G}(\mathcal{O}_v)\times M$ where $M$ is normally generated by one element, Theorem~\ref{theorem::3.1} is proved with $C_1\leqslant C_4+2$.

The remaining case is when ${\bf G}={\rm SL}_2$.
In this situation, we pick a prime $v_0$ of $\mathcal{O}$ and work with the arithmetic group ${\rm SL}_2(\mathcal{O}[{1 \over v_0}])$ which enjoys the congruence subgroup property: we can then use the surjective map with trivial kernel $\widehat{{\rm SL}_2(\mathcal{O}[{1 \over v_0}])} \rightarrow {\rm SL}_2(\widehat{\mathcal{O}[{1 \over v_0}]}) = \prod_{v \neq v_0} {\rm SL}_2(\mathcal{O}_v)$ as before, in order to see that the family $\{ {\bf G}(\mathcal{O}_v) \}_{v \neq v_0}$ is boundedly presented. 
At last, the single remaining group ${\rm SL}_2(\mathcal{O}_{v_0})$ does not cause any problem since it is analytic over $k_{v_0}$ and therefore finitely presented. 
\qed 
\vskip 5mm

In order to prove Theorem~\ref{theorem::3.2} we will need the following statement.

\begin{lemma}
\label{lemma::splitting}
Let G be a finite $p$-group that is an extension of $U$ by $N$
$$1\rightarrow N\rightarrow G\rightarrow U\rightarrow 1.$$
Suppose further that $N$ is abelian, the action of $U$ on $N$ is fixed,
and let $G_0$ be such an extension that splits. Then $d(G)\leqslant d(G_0)$. 
\end{lemma}
\begin{proof}
 The minimal number of generators of a $p$-group $H$ is $\dim (H/[H,H]H^p)$.  
 We can therefore assume that $N$ is an elementary abelian $p$-group.  Moreover, if $H$ is either  $G$ or $G_0$,  we can factor out the subgroup  $[N,H]$, and so assume that the action of  $U$ on $N$ is trivial. 
      In the split case we are looking now at $N\times U$,  and so  the number of generators  is $d(G_0)=d(N)+d(U)$.  
      This is certainly an upper bound for the number of generators of  the non-split case.  
      The latter will indeed be smaller in general as can be seen easily in any non-abelian $p$-group with elementary abelian centre where the number of generators of such $H$  is strictly less than 
 $d(Z(H)) + d (H/Z(H))$  since $Z(H)$ has a non-trivial intersection with the Frattini subgroup $\Phi(H)$. 
\end{proof}

And now on with the proof of Theorem~\ref{theorem::3.2}.

\vskip 5mm

\emph{Proof of Theorem~\ref{theorem::3.2}}.
The group $P$ is an extension of the Sylow $p_v$-subgroup $U$ of ${\bf G}(\mathbf{F}_{p_v^{a_v}})$ by the first congruence subgroup $N:={\rm Ker}\bigl( {\bf G}(\mathcal{O}_v)\rightarrow {\bf G}(\mathbf{F}_{p_v^{a_v}}) \bigr)$ of ${\bf G}(\mathcal{O}_v)$.
This group $N$ is a uniform powerful group whose Frattini subgroup $\Phi(N)$ is equal to the second congruence subgroup of ${\bf G}(\mathcal{O}_v)$.
In fact $N/\Phi(N)$ can be identified with ${\rm Lie}({\bf G})(\mathbf{F}_{p_v^{a_v}})$ and ${\bf G}(\mathcal{O}_v)/N\cong{\bf G}(\mathbf{F}_{p_v^{a_v}})$ acts on $N/\Phi(N)$ by the natural adjoint  action of ${\bf G}(\mathbf{F}_{p_v^{a_v}})$ on its Lie algebra ${\rm Lie}({\bf G})(\mathbf{F}_{p_v^{a_v}})$ \cite[Lemma 5.2]{Weisfeiler}. 
It follows that $d(P)$ is equal to $d\bigl( P/\Phi(N) \bigr)$ and $P/\Phi(N)$ is an extension of ${\rm Lie}({\bf G})(\mathbf{F}_{p_v^{a_v}})$ by $U$. 

Now if $\widetilde{P}$ is the corresponding Sylow pro-$p$ subgroup in the group ${\bf G}(\mathbf{F}_{p_v^{a_v}}[[t]])$ over the characteristic $p$ ring $\mathbf{F}_{p_v^{a_v}}[[t]]$, then $\widetilde{P}$ is the split extension of $U$ with $\widetilde{N}$, where $\widetilde{N}$ is the first congruence subgroup there, i.e., 
$\widetilde{N}={\rm Ker}\bigl( {\bf G}(\mathbf{F}_{p_v^{a_v}}[[t]])\rightarrow {\bf G}(\mathbf{F}_{p_v^{a_v}}) \bigr)$.
Also there, the Frattini subgroup $\Phi(\widetilde{N})$ of $\widetilde{N}$ is the second congruence subgroup of ${\bf G}(\mathbf{F}_{p_v^{a_v}}[[t]])$ and we have 
$d(\widetilde{P})=d\bigl( \widetilde{P}/\Phi(\widetilde{N}) \bigr)$.
In this situation $\widetilde{P}/\Phi(\widetilde{N})$ is also an extension of ${\rm Lie}({\bf G})(\mathbf{F}_{p_v^{a_v}})$ by $U$ (via the adjoint action), but this time the extension splits.

We may therefore apply Lemma~\ref{lemma::splitting} to deduce that
$$d(P) = d\bigl( P/\Phi(N) \bigr) \leqslant d\bigl( \widetilde{P}/\Phi(\widetilde{N}) \bigr) = d(\widetilde{P}).$$

In \S 2, using the Kac-Moody methods,  we showed that $d(\widetilde{P})=a_v(l+1)$ for $l\geqslant 3$ and $p_v^{a_v}\geqslant 16$.
But in fact, as was shown in \cite[Corollary 2.5]{kn::CR}, this result is true for all $l$ as long as $p>2$ for $\mathbf{G}$ being $A_1$, $B_l, C_l$ or $F_4$ and $p>3$ for $\mathbf{G}=G_2$.
Hence,  under our hypotheses, $d(P)\leqslant a_v(l+1)$. 
On the other hand, as $P$ surjects onto $U$, we have $d(P)\geqslant d(U)=a_vl$. 

Observe that $N$ is a powerful uniform group on $a_v\cdot\dim({\bf G})$ generators and hence can be presented by $C'a_v^2\dim({\bf G})^2$ relations \cite[Theorem 4.35]{kn::DDMS}.
The group $U$ has a presentation with $a_vl$ generators and at most $C^{''}a_v^2l^2$ relations.
Indeed, $U$ is a Sylow $p$-subgroup of $ {\bf G}(\mathbf{F}_{p_v^{a_v}})$.
But the latter group is a special case of Kac-Moody group, namely a Chevalley group over a finite field.
Thus in particular, $U=U_+$ for this class of groups and so  if $l\geqslant 3$, Proposition~\ref{prop::presentation_of_Gamma}  holds for $U$ implying the desired estimate.
If $l\leqslant 2$, then $U$ is a group on  $la_v$ generators and of order $p_v^{a_vm}$ with $m\leqslant 6$ by \cite[Theorem 3.3.1]{kn::GLS3}.
Now \cite[Proposition 3.4.1]{kn::LSe} implies that $U$ has a presentation on those generators with at most $12a_v^2$ relations.
Thus for all $l$, $U$ has a presentation  with $a_vl$ generators and at most $C^{''}a_v^2l^2$ relations.
Now we may use P.~Hall's lemma to glue together the presentations of $N$ and $U$ to obtain a presentation of $P$. As a result we conclude that $P$ has a presentation with at most $a_v(l+1)$ generators and at most $C_3a_v^2l^4$ relations, as claimed. 

The lower bound on the number of relations now follows from the Golod-Shafarevich inequality \cite[Proposition 5.4]{kn::LS}.
\qed

\section{Characteristic $0$: the global case}
\label{s - char 0}

In this section, we  prove Theorem~\ref{theorem::3.3}.
The proof follows the line of \cite{kn::Cap} where Theorem 1.1 is proved.
Some results from $K$-theory are needed along the way.

\subsection{Definitions and facts from $K$-theory}
\label{ss - K-theory}

In what follows, $\Phi$ is a reduced, irreducible root system and $R$ a commutative unit ring.
We denote by ${\rm St}_\Phi(R)$ the Steinberg group of type $\Phi$ over $R$, 
which is a group defined by a presentation based on the root subgroups and commutator relation between them.

If $G_\Phi$ is the simply connected Chevalley group scheme of type $\Phi$, there is a natural group homomorphism ${\rm St}_\Phi(R) \to G_\Phi(R)$ whose image is precisely the subgroup $E_\Phi(R)$ generated by the root groups $U_a(R)$, $a \in \Phi$, with respect to the standard maximal split torus.
By the very definition of the low-degree $K$-groups of a root system and a ring, we have an exact sequence:

\smallskip

\centerline{$(\dagger) \quad 1 \to {\rm K}_2(\Phi,R) \to {\rm St}_\Phi(R) \to G_\Phi(R) \to {\rm K}_1(\Phi,R) \to 1,$}

\smallskip

\noindent where (the image of) ${\rm K}_2(\Phi,R)$ is central in ${\rm St}_\Phi(R)$. 
Roughly speaking, ${\rm K}_1(\Phi,R)$ measures the failure for $G_\Phi(R)$ to be generated by its unipotent elements.
When ${\rm K}_1(\Phi,R)$ is trivial, the group ${\rm St}_{\Phi}(R)$ is a central extension of the Chevalley group $G_\Phi(R)$.

\vskip 5mm
In the proof below, an important argument is the fact that the Steinberg version of a celebrated theorem of Curtis and Tits holds over rings.
This result was announced almost 40 years ago by R.K.~Dennis and M.R.~Stein \cite[Th. B]{kn::DS}.
Our reference is D.~Allcock's recent paper \cite{kn::A1} which contains a vast generalization of Curtis-Tits amalgamation theorem to Steinberg groups associated to some classes of  Kac-Moody groups. 

\begin{theorem}
Let $\Phi$ be a reduced irreducible root system of rank $l \geqslant 2$ 
with $\Pi=\{\alpha_1, ... , \alpha_l\}$ a system of simple roots in $\Phi$. 
For each pair of integers $\{ i , j \}$ with $1 \leqslant i,j \leqslant l$, denote by $\Phi_{ij}$ the subsystem of $\Phi$ spanned by $\alpha_i$ and $\alpha_j$.
Let $R$ be a commutative ring with $1$, and define $\Sigma(\Phi, R)$ to be the group
with generators $x_{\alpha}(t)$ for $\alpha\in\bigcup\Phi_{ij}$ and $t\in R$ subject to the relations:

\begin{enumerate}
\item[{\rm 1.}] $x_{\alpha}(s)x_{\alpha}(t)=x_{\alpha}(s+t)$, $\alpha\in\bigcup\Phi_{ij}$, $s,t\in R$,
\item[{\rm 2.}] $[x_{\alpha}(s), x_{\beta}(t)]=\prod x_{a\alpha+b\beta}(N_{\alpha\beta a b}s^at^b)$
\end{enumerate}

\noindent for all  $\alpha,\beta\in \Phi_{ij}$ for some $1 \leqslant i,j \leqslant l$ such that $\alpha+\beta\neq 0$, where the product and the integers $N_{\alpha\beta a b}$  are as in {\rm (R2)}~of {\rm \cite[3.7]{kn::St1}}.  
Then $\Sigma(\Phi, R)\cong {\rm {\rm St}_\Phi}(R)$.
\end{theorem}

\subsection{Proof of Theorem~\ref{theorem::3.3}}
\label{ss - proof}

The proof of Theorem~\ref{theorem::3.3} is now as follows.
\begin{description}
\item[Step I:] The groups $G_\Phi(\mathcal{O})$ are all arithmetic groups and hence are finitely presented \cite{RaghunathanFP}.
So  in particular  the groups $G_\Phi(\mathcal{O})$, for $\Phi$ of rank $l \leqslant 5$, are all presented by at most $C_4'(k)$ generators and relations. As there are only finitely many roots systems of bounded rank, we have to deal from now on only with those of degree at least $6$.

\item[Step II:] Recall that ${\rm Sym}(n)$  and ${\rm Alt}(n)$ have presentations with a bounded number of generators and relations independent of $n$ \cite{kn::GKKL1}.
This was used (together with the fact that ${\rm SL}_4(\mathbf{Z})$ is finitely presented) in \cite{kn::GKKL3} to show that the groups ${\rm SL}_n(\mathbf{Z})$, for $n\geqslant 6$,  are boundedly presented.
The proof works word to word to show that ${\rm St}_{A_l}(\mathcal{O})$ are boundedly presented for $l\geqslant 6$. 

%

\item[Step III:] We now show that every Dynkin diagram $\Phi$ of rank at least $6$ can be covered by three subdiagrams $\Phi_i$ (for $i=1,2,3$) such that:
\begin{enumerate}
\item  Each $\Phi_i$ has  at most two connected components,  and every connected component is either of type $A_l$ (for $l\geqslant 2$) or of rank at most $5$,
\item Every two nodes of $\Phi$ belong to at least one of the $\Phi_i$'s, and
\item Each intersection $\Phi_i\cap\Phi_j$ (for $i\neq j$) has at most two components, each either empty or of type $A_l$ (for $l\geqslant 2$), $B_3$ or $C_3$.
\end{enumerate} 

Let us list explicitly $\Phi_i$, $i=1,2,3$, in each particular case:
\begin{itemize}
\item For $B_l$, $l\geqslant 6$, take $\Phi_1=A_{l-1}$, $\Phi_2=B_3\sqcup A_{l-4}$ and $\Phi_3=\varnothing$,
\item For $C_l$, $l\geqslant 6$, take $\Phi_1=A_{l-1}$, $\Phi_2=C_3\sqcup A_{l-4}$ and $\Phi_3=\varnothing$,
\item For $D_l$, $l\geqslant 6$, take $\Phi_1=A_{l-1}$, $\Phi_2=A_{l-1}$ (different from $\Phi_1$) and $\Phi_3=A_3$ (the one containing the ``fork"),
\item For $E_l$, $l\in\{6, 7, 8\}$, take $\Phi_1=A_{l-1}$, $\Phi_2= A_4$  (the one containing  one of the end-nodes and the vertex missed by $\Phi_1$) and $\Phi_3=A_{l-2}$ (the one containing the other end-node and containing the vertex missed by $\Phi_1$).
\end{itemize}

\item[Step IV:] Given $\Phi=\bigcup_{i=1}^3\Phi_i$ as in the previous step, the Curtis-Tits Theorem guarantees that a presentation for ${\rm St}_\Phi(\mathcal{O})$ is obtained by taking the union of the presentation of ${\rm St}_{\Phi_i}(\mathcal{O})$ and gluing them along the intersections. As we have arranged that all the ${\rm St}_{\Phi_i}(\mathcal{O})$  as well as their intersections have bounded presentations, the same now applies to ${\rm St}_\Phi(\mathcal{O})$.

\item[Step V:]  We now use the stability results to deduce that by adding a bounded number of relations to the presentation of ${\rm St}_\Phi(\mathcal{O})$, we obtain a bounded presentation for $G_\Phi(\mathcal{O})$ as promised, where the bound depends on $\mathcal{O}$ only.

First of all notice that by the results of Mennicke, Bass-Milnor-Serre and Matsumoto (cf. Theorem 7.4 of \cite{kn::S}), the group $ {\rm K}_1(\Phi,\mathcal{O})$ is trivial  when  $l\geqslant 2$, and therefore definitely in our case.

Thus in view of $(\dagger)$ and the result  just obtained in Step IV, to finish our proof it remains to show that for all $\Phi$ of rank at least  $6$, the corresponding family of groups $\{K_2(\Phi, \mathcal{O})\}_\Phi$ is boundedly generated. 

To do that, let us recall  the stability results of Dennis, Van der Kallen and Stein that are summarised in \cite[Theorem 7.5]{kn::S} and in the example following it. Applied to our situation they imply  that for $l\geqslant 6$ and $\Phi_l=A_l, B_l, C_l$ or $D_l$, the maps $$K_2(\Phi_6, \mathcal{O})\rightarrow K_2(\Phi_l,\mathcal{O})$$ are surjective as well as for $l=6,7,8$, are the maps
$$K_2(A_6,\mathcal{O})\rightarrow K_2(E_l,\mathcal{O}).$$
It follows that $K_2(\Phi_l,\mathcal{O})$, $l\geqslant 6$, is boundedly generated provided we can show that $K_2(\Phi_6, \mathcal{O})$ are finitely generated. 
But this follows immediately as ${\rm St}_{\Phi_6}(\mathcal{O})$ is boundedly presented (as we showed in the previous step) while $G_{\Phi_6}(\mathcal{O})$ is finitely presented (as an arithmetic group):~
this implies that $K_2(\Phi_6, \mathcal{O})$ is finitely generated as a normal subgroup; being central, it is also finitely generated as a group. 
\end{description}

\begin{remark}
As pointed out above in Remark~\ref{remark::34}, the constant $C_4(k)$ in Theorem~\ref{theorem::3.3} does depend on $k$. In fact we observed there that there exist sequences of fields $k_i$ with $|k_i:\mathbf{Q}|\rightarrow \infty$ such that $C_4(k_i)\geqslant c_0\log|k_i:\mathbf{Q}|$ for some absolute constant $c_0$. One can carefully analyse the proof of Theorem~\ref{theorem::3.3} to deduce that $C_4(k)\leqslant C_0|k:\mathbf{Q}|^2$ (when $l\geqslant 6$).
\end{remark}

\bigskip

\noindent {\bf Acknowledgements.}~The authors would like to thank the anonymous referee for his/her detailed comments that helped to improve the paper.

The first author was supported by the Leverhulme Research Grant RPG-2014-106.
The second author wishes to acknowledge support by the ISF, the NSF, Dr.~Max R\"ossler, the Walter Haefner Foundation and the ETH Foundation. 
The third author would like to thank Amaury Thuillier for many useful discussions and was supported by the ANR GDSous/GSG project ANR-12-BS01-0003 and the Alexander von Humboldt Foundation.

\vspace{1cm}

{\it Inna Capdeboscq}

Mathematics Institute

University of Warwick

Zeeman Building Coventry CV4 7AL -- United Kingdom

{\tt  I.Capdeboscq@warwick.ac.uk}

\medskip

{\it Alexander Lubotzky}

Einstein institute of mathematics

Hebrew University

Givat Ram, Jerusalem 91904 -- Israel

{\tt alex.lubotzky@mail.huji.ac.il}

\medskip

{\it Bertrand R\'emy}

CMLS

\'Ecole polytechnique, CNRS

Universit\'e Paris-Saclay

91128 Palaiseau cedex -- France

{\tt bertrand.remy@polytechnique.edu}

\end{document}